\documentclass[secthm,seceqn,amsthm,ussrhead,11pt]{amsart}

\usepackage{amssymb,amsmath,amsthm,latexsym,booktabs}
\usepackage{marginnote}
\usepackage{amsmath}
\usepackage[T1]{fontenc}

\usepackage{color}
\usepackage{pifont}

\theoremstyle{definition}
\newtheorem{definition}{Definition}[section]

\newtheorem{remark}[definition]{Remark}

\theoremstyle{plain}
\newtheorem{lemma}[definition]{Lemma}
\newtheorem{proposition}[definition]{Proposition}
\newtheorem{theorem}[definition]{Theorem}
\newtheorem{corollary}[definition]{Corollary}

\newfont{\hueca}{msbm10}
\def\hu #1{\hbox{\hueca #1}}\def\hu #1{\hbox{\hueca #1}}

\newcommand{\ad}{{\mathrm{ ad}}}
\newcommand{\ann}{{\mathrm{ Ann }}}
\begin{document}

{\Large Split Regular  $Hom$-Leibniz Color $3$-Algebras
\footnote{The work was supported by
FAPESP 16/16445-0, 17/15437-6;
RFBR 18-31-00001 and
the Presidents Programme Support of Young Russian Scientists (MK-1378.2017.1).}
}

\

\

\medskip

\medskip
\textbf{Ivan Kaygorodov$^{a,}$\footnote{Corresponding Author: kaygorodov.ivan@gmail.com}  \& Yury Popov$^{b}$}
\medskip

\

\

{\tiny
$^{a}$ Universidade Federal do ABC, CMCC, Santo Andr\'{e}, Brazil.

$^{b}$ Universidade Estadual de Campinas, Campinas, Brazil.

\ \smallskip

    E-mail addresses:\smallskip

\

    Ivan Kaygorodov (kaygorodov.ivan@gmail.com),

    \smallskip

    Yury Popov (yuri.ppv@gmail.com).

}

\

\

{\bf Abstract.}
We introduce  {and describe} the class of
split regular $Hom$-Leibniz color $3$-algebras
as the natural extension of the class of
split Lie algebras,
split Lie superalgebras,
split Lie color algebras,
split regular $Hom$-Lie algebras,
split regular $Hom$-Lie superalgebras,
split regular $Hom$-Lie color algebras,
split Leibniz algebras,
split Leibniz superalgebras,
split Leibniz color algebras,
split regular $Hom$-Leibniz algebras,
split regular $Hom$-Leibniz superalgebras,
split regular $Hom$-Leibniz color algebras,
split Lie $3$-algebras,
split Lie $3$-superalgebras,
split Lie color $3$-algebras,
split regular $Hom$-Lie $3$-algebras,
split regular $Hom$-Lie $3$-superalgebras,
split regular $Hom$-Lie color $3$-algebras,
split Lie triple systems,
split Lie triple supersystems,
split Lie triple color systems,
split regular $Hom$-Lie triple systems,
split regular $Hom$-Lie triple supersystems,
split regular $Hom$-Lie triple color systems,
split Leibniz $3$-algebras,
split Leibniz $3$-superalgebras,
split Leibniz color $3$-algebras,
split regular $Hom$-Leibniz $3$-algebras,
split regular $Hom$-Leibniz $3$-superalgebras,
and some other algebras.

More precisely, we show that any of such split
 regular $Hom$-Leibniz  color $3$-algebras $T$ is of the form
  ${T}={\mathcal U} +\sum\limits_{j}I_{j}$, with ${\mathcal U}$
  a subspace of the $0$-root space ${T}_0$, and
   $I_{j}$  {an ideal} of $T$ satisfying {for} $j\neq k:$
\[[{ T},I_j,I_k]+[I_j,{ T},I_k]+[I_j,I_k,T]=0.\]
     {Moreover, if} $T$ is of maximal length, we characterize the simplicity of $T$ in terms of a connectivity property in its set of non-zero roots.

     \medskip

\


\



\newpage
\section{Introduction}

$n$-Ary algebras have been applied in mathematical physics,
in the study of the supersymmetry, Bagger-Lambert theory or Nambu mechanics (see, for examples, papers \cite{[1],[2],[3],[4]}).
On the other hand, there is a pure algebraic interest in the study of ternary algebras \cite{ck10,ck14,kg,kps,poj} and, more generally, $n$-ary Leibniz algebras \cite{casas,casas2,poj2} and general $n$-ary algebras \cite{BCKS,kp}.

In this paper we study the class of $Hom$-Leibniz color $3$-algebras (of which the classes of $Hom$-Lie color $3$-algebras and
$Hom$-Lie triple color systems are subclasses) and Leibniz color 3-algebras with an automorphism.

The study of $Hom$-structures began in the paper of Hartwig, Larsson and  Silvestrov \cite{hom}.
The notion of $Hom$-Lie triple systems was introduced in \cite{[20]}. 
In the paper \cite{[28]} Yau gave a general construction method of $Hom$-type algebras starting from usual algebras and a twisting self-map. For information on various types of $Hom$-algebras see \cite{[6], bak, BKP,chen2,[20],chen}.

  In the present paper we study the structure of split regular $Hom$-Leibniz $3$-algebras of arbitrary dimension and over an arbitrary base field ${\hu K}$.
  Split structures appeared first in the classical theory of (finite dimensional) Lie algebras but have been extended to more general settings
  like, for example,
  Leibniz algebras \cite{AJC},
  Poisson algebras,
  Leibniz superalgebras \cite{AJCS1},
  regular $Hom$-Lie algebras \cite{AraCM1},
  regular $Hom$-Lie superalgebras \cite{sphomliesu},
  regular $Hom$-Lie color algebras \cite{Cao3},
  regular $Hom$-Poisson algebras \cite{AraCM2},
  regular $Hom$-Leibniz algebras \cite{Cao4},
  regular $BiHom$-Lie algebras \cite{AJCS2},
  regular $BiHom$-Lie superalgebras \cite{spbihomliesu},
  among many others.
    As for the study of split ternary structures,  {see \cite{AJ2, Twisted, AJF2} for Lie triple systems, twisted inner derivation triple systems, Lie $3$-algebras \cite{AJF2}, Leibniz $3$-algebras \cite{3L} and \cite{Cao1, Cao2} for Leibniz triple systems.}

This paper is organized as follows:
 in  Section 2 we recall the main definitions and results related to the $3$-algebras theory. For $T$ a $Hom$-Leibniz color $3$-algebra or a Leibniz color 3-algebra with an automorphism we construct its multiplication algebra $\mathfrak{L},$ which is a Lie color algebra with an automorphism, and its standard envelope, in terms of which we give the definition of a split $Hom$-Leibniz color $3$-algebra and of a split Leibniz color 3-algebra with an automorphism, and show how the concepts are related.
 In  Section 3 we develop roots connections techniques in this framework and apply it to get certain decompositions theorem for the algebras above. We prove that an arbitrary split regular $Hom$-Leibniz color $3$-algebra $T$ can be decomposed as
${T}={\mathcal U} +\sum\limits_{j}I_{j}$, with ${\mathcal U}$ a subspace of the $0$-root space ${ T}_0$ and $I_{j}$  are ideals of $T$ satisfying
\[[{T},I_j,I_k]+[I_j,{T},I_k]+[I_j,I_k,T] = 0, \, \mbox{for} \, j \neq k.\]
Finally, in  Section 4 we see how the above decomposition of $T$ can be used to obtain a similar decomposition of $\mathfrak{L}.$ 

${\emph Remark.}$
Note that this paper is the first attempt to extend the theory of split algebras to  $n$-ary $Hom$-algebras $(n>2).$ Up until now, only split $Hom$-algebras and split 3-algebras have been studied.

\newpage
\section{Preliminaries}

\subsection{Color algebras}

In this subsection we recall the basic notions related to color algebras and maps on them.

\begin{definition}{
Let $\mathbb{K}$ be a field and $\mathbb{G}$ be an abelian group.
A map $\epsilon: \mathbb{G} \times \mathbb{G} \rightarrow \mathbb{K}^{\times}$ is called a bicharacter on $\mathbb{G}$ if the
following relations hold for all $f,g,h \in \mathbb{G}:$

(1) $\epsilon(f,g+h)=\epsilon(f,g)\epsilon (f,h);$

(2) $\epsilon(g+h,f)=\epsilon(g,f)\epsilon(h,f);$

(3) $\epsilon(g,h)\epsilon(h,g)=1.$}
\end{definition}

\begin{definition}
A $\mathbb{G}$-graded color $n$-ary algebra $T$ is a vector space $T=\bigoplus_{g\in \mathbb{G}} T_g$ with an $n$-linear map
$[\cdot , \ldots, \cdot ]: T \times \ldots \times T \rightarrow T$  satisfying
\[[T_{\theta_1},  \ldots, T_{\theta_n}] \subseteq T_{\theta_1+ \ldots +\theta_n}, \ \theta_i \in \mathbb{G}.\]
\end{definition}

Let $T=\bigoplus_{g\in \mathbb{G}} T_g$ be a color algebra.
An element $x$ is called homogeneous of degree $t \in \mathbb{G}$ if $x \in T_t$. We denote this by $hg(x)=t$.
A $\mathbb{G}$-grading on $T$ induces a $\mathbb{G}$-grading on the space of linear maps on $T:$ a linear map $D$ on $T$ is homogeneous of degree $t$ if $D(T_g) \subseteq T_{g+t}$ for all $g \in \mathbb{G}.$
We denote this by $hg(D)=t$. If $D$ is homogeneous of degree $t=0$ then $D$ is said to be even. From now on, unless stated otherwise, we assume that all elements and maps are homogeneous. Let $\epsilon$ be a bicharacter on $\mathbb{G}$. For two homogeneous elements $a$ and $b$ we set $\epsilon(a,b):=\epsilon(hg(a),hg(b)).$



\subsection{Color $Hom$-algebras.}
In this subsection we consider  $Hom$-algebraic structures that are twisted versions of the original algebraic structures, and linear maps on them.

\begin{definition}
A $Hom$-algebra is an algebra $A$ with a fixed linear map $\phi$. A homomorphism of two $Hom$-algebras $(A,\phi)$ and $(B,\psi)$ is a map $\varphi$ which is an algebra homomorphism between $A$ and $B$ and $\varphi\phi = \psi\varphi.$
\end{definition}

\begin{remark}
The most general definition says that an $n$-ary $Hom$-algebra is an $n$-ary algebra $A$ equipped with a certain family of linear maps $\{\phi_i\}_{i \in I}.$ However, it is hard to work in such a general context. Therefore, in this paper for convenience we will only consider the case where all maps $\phi_i$ are equal to one map $\phi$ which is also a homomorphism of the underlying algebra structure of $A$.
Such algebras are called multiplicative $n$-ary $Hom$-algebras. From now on we will work only with multiplicative $Hom$-algebras (the exact formulations will be given later in the paper). Also from now on we will assume that $\phi$ is an automorphism of $A$. In this case $A$ is called a regular $Hom$-algebra.
\end{remark}

The definitions of color and $Hom$-algebras can be unified into one:
\begin{definition}
A color $Hom$-algebra $A$ is a color algebra $(A,\epsilon)$ with a fixed even linear map $\phi$. A homomorphism of two color $Hom$-algebras $(A,\mathbb{G}, \epsilon_A,\phi)$ and $(B,\mathbb{G},\epsilon_B,\psi)$ is an even map $\varphi: A\to B$ that is an algebra homomorphism between $A$ and $B$ and $\varphi\phi = \psi\varphi.$
\end{definition}




\subsection{Color $Hom$-$\Omega$ algebras.}

In this subsection we give definitions and study the basic properties of $Hom$-Leibniz color (binary and ternary) algebras.

First we discuss briefly the color $Hom$-algebras defined by identities. We note here that given a variety of ($n$-ary) algebras $\Omega$ defined by a family of identities $\{f_i\}_{i \in I}$ we may obtain the identities $\{(f_i)_{co}^{Hom}\}_{i\in I}$ of the corresponding variety of color $Hom$-$\Omega$ algebras by twisting the defining identities by homomorphisms, usually called the twisting maps. The details of this process can be found in the papers \cite{BKP, CM}.

As an example of this process, we give the definitions of $Hom$-Leibniz color (3-)algebras, and their important subclasses.  As usual, let $\mathbb{G}$ be an abelian group and $\epsilon$ be a bicharacter on $\mathbb{G}.$ From now on we only work with regular algebras.

\begin{definition}
\label{colorleibbin}
A regular $Hom$-Leibniz color algebra $(L,[\cdot,\cdot], \epsilon, \phi)$ is a $\mathbb{G}$-graded vector space $L$ with an even multiplication $[\cdot,\cdot]$ and an even automorphism $\phi$ satisfying
\[[\phi(x),[y,z]] =[[x,y],\phi(z)] + \epsilon(x,y)[\phi(y),[x,z]].\]
If additionally the identity
\[[x,y]=-\epsilon(x,y)[y,x]\]
holds in $L$ (that is, the multiplication in $T$ is color-anticommutative), then $L$ is called a regular $Hom$-Lie color algebra.
\end{definition}

\begin{definition}\label{colorleibter}
A regular $Hom$-Leibniz color $3$-algebra $(T,[\cdot,\cdot], \epsilon, \phi)$ is a $\mathbb{G}$-graded vector space $T$ with a bicharacter $\epsilon$, an even trilinear map $[\cdot,\cdot,\cdot]$ and an even automorphism $\phi$ satisfying
\begin{multline}
\label{DI}
[\phi(x_1),\phi(x_2),[y_1,y_2,y_3]]=  \\
[[x_1,x_2,y_1],\phi(y_2),\phi(y_3)] + \epsilon(x_1 + x_2,y_1)[\phi(y_1),[x_1,x_2,y_2],\phi(y_3)] + \epsilon(x_1+x_2,y_1+y_2)[\phi(y_1),\phi(y_2),[x_1,x_2,y_3]].    
\end{multline}

Additionally, if the identities 
\[[x_1,x_2,x_3]=-\epsilon(x_1,x_2)[x_2,x_1,x_3],\ [x_1,x_2,x_3]=-\epsilon(x_2,x_3)[x_1,x_3,x_2]\]
hold in $T$, then $T$ is called a regular $Hom$-Lie color $3$-algebra, and if the identities
\begin{gather*}
[x_1,x_2,x_3]=-\epsilon(x_1,x_2)[x_2,x_1,x_3],\\ \epsilon(x_3,x_1)[x_1,x_2,x_3]+ \epsilon(x_1,x_2)[x_2,x_3,x_1]+  \epsilon(x_2,x_3)[x_3,x_1,x_2]=0
\end{gather*}
hold in $T,$ then $T$ is called a regular $Hom$-Lie triple color system.
\end{definition}


\begin{remark} In particular, if $\mathbb{G} = \mathbb{Z}_2,$ we get the definition of a $Hom$-Leibniz (3-)superalgebra, and if $\mathbb{G}$ is the trivial group $\{1\}$ then we get the definition of a $Hom$-Leibniz ($3$-)algebra.
If $\phi$ is the identity automorphism on $T$ then we get the definition of a Leibniz color $3$-algebra, as the defining identity turns
\begin{equation} \label{DII}
\begin{aligned}
&[x_1,x_2,[y_1,y_2,y_3]]  = \\
&[[x_1,x_2,y_1],y_2,y_3] + \epsilon(x_1 + x_2,y_1)[y_1,[x_1,x_2,y_2],y_3] + \epsilon(x_1+x_2,y_1+y_2)[y_1,y_2,[x_1,x_2,y_3]].
\end{aligned}
\end{equation}
\end{remark}

As an example, let $T$ be any $\mathbb{G}$-graded space. By ${\rm End}(T)$ we denote the set of all linear maps on $T$. One easily checks that ${\rm End}(T)$ with the grading ${\rm End}(T) = \bigoplus_{g \in \mathbb{G}}{\rm End}_g(T)$ endowed with the color bracket 
\begin{equation}
\label{colorcom}
[D_1,D_2] = D_1D_2 - \epsilon(D_1,D_2)D_2D_1,
\end{equation}
is a color Lie algebra.

\begin{definition}\label{cls}
Let $T$ be a  $Hom$-Leibniz color $3$-algebra.
An ideal of  $T$ is a $\phi$-invariant subspace $I$ such that
\[[I,T,T] + [T,I,T] + [T,T,I] \subseteq I.\]
The annihilator of $T$ is the set
\[
\ann (T) = \{ x \in T: [x, T, T] + [T, x, T] + [T, T, x] = 0\}.
\]
It is straightforward to check that $\ann (T)$ is an ideal of $T$.
\end{definition}

\medskip

So, our main objects will be regular $Hom$-Leibniz color $3$-algebras. Any such algebra $T$ is given by the following data: $T = (T,[\cdot,\cdot,\cdot],\epsilon,\phi),$ where $\epsilon$ is a $\mathbb{G}$-bicharacter on $T$ and $\phi$ is an automorphism such that the identity (\ref{DI}) holds.

Consider also the category of Leibniz color 3-algebras with automorphisms. Any such algebra $T$ is given by the following data: $T = (T,[\cdot,\cdot,\cdot],\epsilon,\phi),$ where $\epsilon$ is a $\mathbb{G}$-bicharacter on $T$ such that the identity (\ref{DII}) holds and $\phi$ is an automorphism on $T.$ The morphisms in this category are the maps respecting the algebra structure, grading and commuting with automorphisms.

Note that the second category is not a subset of the first: the defining relations are different! However, these categories are isomorphic. Indeed, given a Leibniz color $3$-algebra with automorphism $T = (T,[\cdot,\cdot,\cdot],\epsilon,\phi),$ consider an algebra $T^{\phi} = (T,[\cdot,\cdot,\cdot]^{\phi},\phi,\epsilon)$, where
\[[x,y,z]^{\phi} = \phi([x,y,z]).\]
One can check that $T^{\phi}$ is a regular $Hom$-Leibniz color $3$-algebra.
Moreover, this construction is universal in the following sense: if $T = (T,[\cdot,\cdot,\cdot],\phi,\epsilon)$ is a regular $Hom$-Leibniz color $3$-algebra, then one easily checks that $T^{\phi^{-1}}=(T,[\cdot,\cdot,\cdot]^{\phi^{-1}},\epsilon)$, where
\[[x,y,z]^{\phi^{-1}} = \phi^{-1}([x,y,z]),\]
is a  Leibniz color $3$-algebra with automorphism $\phi$, and this construction is clearly inverse to the one above.
Hence, the category of regular  $Hom$-Leibniz color $3$-algebras is isomorphic to the category of Leibniz color $3$-algebras with an automorphism. Occasionally, we will refer to the passage $T \to T^{\phi^{-1}}$ as ''dehomification''.

\medskip

The ideal $J$ generated by the set
\[
\big \{ [x,y,z]+\epsilon(x,y)[y,x,z], \  [x,y,z]+\epsilon(y,z)[x,z,y]: x, y, z \in T \big \}
\]
plays an important role in the study of  $Hom$-Leibniz color $3$-algebras, since $T$ is a  $Hom$-Lie color $3$-algebra if and only if $J = \{0\}$.

Following the ideas of Abdykassymova and Dzhumadil'daev \cite{Ab} for Leibniz algebras, and  {of} Cao and Chen \cite{Cao1, Cao2} for Leibniz triple systems, we  {introduce the following notion.}

\begin{definition}\label{Defsimple}
We say that $(T,\phi)$ is a simple $Hom$-Leibniz color $3$-algebra if its triple product is nonzero and its only $\phi$-invariant ideals are $\{0\}$, $J$ and $T$.
\end{definition}
Note that this definition  {is consistent} with the notion of a simple $Hom$-Lie $3$-algebra since $J = \{0\}$ in that case.

\subsection{ { Operator algebras}}

In this subsection, inspired by the ideas of the paper \cite{vand}, we associate with a $Hom$-Leibniz color $3$-algebra $T$ a certain  Lie color algebra $\mathfrak{L}$ of ''derivations'' of $T,$ and analogously for a Leibniz color 3-algebra with an automorphism. This construction will be important in the following discussion.

Let $T = (T,[\cdot,\cdot,\cdot],\epsilon,\phi)$ be a  $Hom$-Leibniz color $3$-algebra.
Consider the space $\mathfrak{L} = {\rm span}\{\ad(x,y): x, y \in T\},$ where $\ad (x,y)(z):=[x,y,z]$ and let $\ell \in \mathfrak{L}.$ Then the identity (\ref{DI}) implies that $\ell$ satisfies the following relation:
\begin{equation}
\label{hom_der}
(\ell\circ\phi^{-1})[y_1,y_2,y_3] = [(\phi^{-1}\circ\ell) y_1,y_2,y_3] + \epsilon(\ell,y_1)[y_1,(\phi^{-1}\circ\ell)y_2,y_3] + \epsilon(\ell,y_1+y_2)[y_1,y_2,(\phi^{-1}\circ\ell)y_3].
\end{equation}

Consider the algebra End$(T)$ with the multiplication
\[[D_1,D_2]_{\phi^{-1}} = D_1\phi^{-1}D_2 - \epsilon(D_1,D_2)D_2\phi^{-1}D_1.\]
This algebra is called the $\phi^{-1}$-homotope of ${\rm End}(T)$ and is in fact isomorphic to End$(T)$ by the mapping $\varphi$ sending $\ell \in $ End$(T)$ to $\phi^{-1}\ell.$ In particular, End$(T)_{\phi^{-1}}$  is a Lie color algebra.

Note then that the relation (\ref{hom_der}) implies that $\mathfrak{L}$ is a subalgebra of the homotope algebra End$(T)_{\phi^{-1}}.$ In particular, the multiplication in $\mathfrak{L}$ is given by
\begin{equation}
\label{AL1}
\big [\ad (a,b), \,\ad (c,d) \big]=
\ad \big([a,b,c]^{\phi^{-1}},d \big)+\epsilon(a+b,c) \ad \big (c,[a,b,d]^{\phi^{-1}} \big)
\end{equation}

Note that the conjugation by $\phi$ remains an automorphism in the End$(T)_{\phi^{-1}}:$
\[ [\phi D_1 \phi^{-1}, \phi D_2 \phi^{-1}]_{\phi^{-1}} = \phi[D_1,D_2]_{\phi^{-1}}\phi^{-1}.\]
Since $\phi$ is an automorphism of $T,$ one can see that the algebra $\mathfrak{L}$ is invariant under this automorphism:
\[ \phi\ad(x,y)\phi^{-1} = \ad(\phi(x),\phi(y)).\]
The algebra $\mathfrak{L}$ also has a $\mathbb{G}$-grading induced by $\mathbb{G}$-grading of $T.$ Hence, $(\mathfrak{L},[\cdot,\cdot]_{\phi^{-1}},\epsilon,\phi\cdot\phi^{-1})$ is a Lie color algebra with an automorphism. This algebra is called the multiplication algebra of $T.$

Note that since every element of ${\mathfrak{L}}$ is of the form $\sum
\ad(x_i, y_i)$, we have
\begin{equation}\label{AL2}
[\ell, \ad(a,b)]= \ad (\phi^{-1}\ell a,b) + \epsilon(\ell,a)\ad (a, \phi^{-1}\ell b)
\end{equation}
for any $a,b \in T$ and $\ell \in {\mathfrak L}$.

\medskip

Now we consider an analogous construction for $T = (T,[\cdot,\cdot,\cdot],\epsilon,\phi)$ a Leibniz color $3$-algebra with automorphism. The difference here is that the construction almost does not depend on the automorphism $\phi,$ so we may suppose that we are doing it for a $Hom$-Leibniz color $3$-algebra with a trivial automorphism $(T,[\cdot,\cdot,\cdot],\epsilon,id),$ thus, it is a particular case of the construction above. The only part of the new construction in which the automorphism $\phi$ participates explicitly is the construction of the automorphism on $\mathfrak{L}.$

Recall that a linear mapping $D$ on $T$ is called a derivation if it satisfies
\[ D[a,b,c] = [D(a),b,c] + \epsilon(D,a)[a,D(b),c] + \epsilon(D,a+b)[a,b,D(c)] \]
for all $a,b,c \in T.$
The identity (\ref{DII}) then says that the {left multiplication operators} $\ad (x,y), x, y \in T$, are derivations.

Moreover, in this case the homotopy above is trivial and the bracket in $\mathfrak{L}$ coincides with the bracket (\ref{colorcom}):
\begin{equation} \label{AL1_color}
\big [\ad (a,b), \,\ad (c,d) \big]= \ad(a,b)\circ\ad(c,d) - \epsilon(a+b,c+d)\ad(c,d) \circ \ad(a,b)=
\end{equation}
\begin{equation*}
\ad \big([a,b,c],d \big)+\epsilon(a+b,c) \ad \big (c,[a,b,d] \big).
\end{equation*}
Analogously to the construction above, $\mathfrak{L}$ has a $\mathbb{G}$-grading induced by $\mathbb{G}$-grading of $T$ and the conjugation by $\phi$ induces an automorphism in $\mathfrak{L}.$
Thus, in this case $(\mathfrak{L},[\cdot,\cdot],\epsilon,\phi\cdot\phi^{-1})$ is a Lie color algebra with an automorphism.
Yet again we call it the multiplication algebra of $T.$

In this case the identity (\ref{AL2}) takes the following simplified form:
\begin{equation}\label{AL2_color}
[\ell, \ad(a,b)]= \ad (\ell a,b) + \epsilon(\ell,a)\ad (a, \ell b)
\end{equation}
for any $a,b \in T$ and $\ell \in {\mathfrak L}$.

\subsection{The standard embedding}
This subsection continues the previous one and is inspired by the paper \cite{vand}. Here, for $T$ a regular $Hom$-Leibniz color $3$-algebra or a Leibniz color $3$-algebra with automorphism we extend the algebra $\mathfrak{L}$ constructed above by the space isomorphic to $T$ and obtain a larger color (but not necessary color Lie) $2$-graded algebra $\mathcal{A}$ containing $T$ in a certain sense. This algebra is called the standard envelope of $T$ and plays a crucial role in our discussion. Finally, we prove that the standard envelopes of the algebras $T$ and its ''dehomification'' $T^{\phi^{-1}}$ are isomorphic, explicitly constructing an isomorphism between them.

Recall that an algebra $A$ is said to be  $2$-graded if there exist two linear subspaces $A^{0}$ and $A^{1}$ of $A$,
called the even and the odd part respectively, such that $A = A^{0}\oplus A^{1}$ and $A^{\alpha }A^{\beta }\subset A^{\alpha +\beta }$ for every $\alpha, \beta \in {\hu Z}_{2}$.
\begin{definition}\label{defembed}
The {\bf standard embedding} of a regular $Hom$-Leibniz color $3$-algebra $(T,[\cdot,\cdot,\cdot],\epsilon,\phi)$ is a color
$2$-graded algebra with an automorphism $({\mathcal A},\epsilon,\Phi)$, where ${\mathcal A}^{0}:= {\mathfrak L}$, ${\mathcal A}^{1}:= T$, the product is given by
\begin{equation}
\label{main}
(\ad(x,y),z) \cdot (\ad(u,v),w) =
\end{equation}
\begin{equation*}
(\ad( [x,y,u]^{\phi^{-1}},v)  + \epsilon(x+y,u)\ad(u, [x,y,v]^{\phi^{-1}})+ \ad(z,w) \ , \    [x,y,w]^{\phi^{-1}}-\epsilon(z, u+v)[u,v,z]^{\phi^{-1}})
\end{equation*}
\end{definition}
the automorphism $\Phi$ by
\begin{equation}
\Phi:
\begin{cases}
x &\mapsto \phi(x),\\
\ad(x,y) &\mapsto \phi\ad(x,y)\phi^{-1} = \ad(\phi(x),\phi(y)),
\end{cases}
\text{ for } x, y \in T,
\end{equation}
and the $\mathbb{G}$-grading is induced by $\mathbb{G}$-gradings of $T$ and $\mathfrak{L}.$

It can be easily verified that the map $\Phi$ is indeed an automorphism of $\mathcal{A}.$

The name ''embedding'' is justified in the following sense: introduce a structure of a Leibniz color 3-algerba in $\mathcal{A}$ by $[x,y,z] = [[x,y],z].$ Consider then a regular $Hom$-Leibniz color $3$-algebra $\mathcal{A}^{\Phi}.$ One can then easily see that the subspace $T \subset A^{\Phi}$ with the induced multiplication, automorphism and $\mathbb{G}$-grading is exactly the algebra $(T,[\cdot,\cdot,\cdot],\epsilon,\phi)$ (for details, see paper \cite{vand}).

\begin{definition}Let now $(T,[\cdot,\cdot,\cdot],\epsilon,\phi)$ be a Leibniz color $3$-algebra with automorphism $\phi.$ Then the standard embedding of $T$ is a color $2$-graded algebra with an automorphism $({\mathcal A},\epsilon,\Phi)$, where ${\mathcal A}^{0}:= {\mathfrak L}$, ${\mathcal A}^{1}:= T$, the product is given by
\begin{equation} \label{mainsimple}
\begin{aligned}  (\ad(x,y),z) \ \cdot \ &(\ad(u,v),w) := \\
& (\ad([x,y,u],v)  + \epsilon(x+y,u) \ad(u,[x,y,v]) + \ad(z,w) \ , \ [x,y,w]- \epsilon(z, u+v)[u,v,z]).\end{aligned}
\end{equation}
and automorphism $\Phi$ and $\mathbb{G}$-grading given as in the previous definition.
\end{definition}

\begin{remark}
Although in both cases ${\mathcal A^0}$ is a  Lie color algebra,
${\mathcal A}$ is not, in general, a ($2$-graded)  Lie color algebra.
However, analogously to the paper \cite{vand} one can prove that if $T$ is a $Hom$-Lie triple color system,
then $\mathcal{A}$ is a Lie  color algebra.
\end{remark}

We will consider the relation between the standard embeddings of a $Hom$-Leibniz  color $3$-algebra $(T,[\cdot,\cdot,\cdot],\phi,\epsilon)$ and its ''dehomification'' $T^{\phi^{-1}}$ defined above. Note first that by the definition of the algebra $T^{\phi^{-1}}$ we have
\begin{equation}
\label{ad_relation}
\ad_{T^{\phi^{-1}}}(x,y) = \phi^{-1}\ad_T(x,y),
\end{equation}
where $x, y \in T$ and $\ad_T(x,y), \  \ad_{T^{\phi^{-1}}}(x,y)$ are the left multiplication operators in the corresponding algebras. Thus, we have
\[\mathfrak{L}(T^{\phi^{-1}}) = \phi^{-1}\mathfrak{L}(T)\]
and this allows us to define us a linear map (recall the isomorphism between End$(T)$ and its homotope End$(T)_{\phi^{-1}}$ of the previous subsection) $\varphi: \mathcal{A}(T) \to \mathcal{A}(T^{\phi^{-1}})$ by
\[\varphi:
\begin{cases}
x &\mapsto  x,\\
\ad_T(x,y) & \mapsto \ad_{T^{\phi^{-1}}}(x,y) = \phi^{-1}\ad_T(x,y)
\end{cases}
\text{ for } x, y \in T.\]
\begin{proposition}
\label{emb_iso}
The map $\varphi$ is an isomorphism of color 2-graded algebras with automorphism.
\end{proposition}
\begin{proof} By construction, $\varphi$ is bijective and respects $\mathbb{G}$-grading and 2-grading. So we only need to verify that $\varphi$ respects multiplication and commutes with automorphisms, which is done by straightforward checking. Let $x, y, z, w \in T.$ Using the definitions of multiplications in the algebras $\mathcal{A}(T)$ and $\mathcal{A}(T^{\phi^{-1}}),$ we have
\[\varphi(\ad_T(x,y))\cdot_{\mathcal{A}(T^{\phi^{-1}})}\varphi(\ad_T(z,w)) = \ad_{T^{\phi^{-1}}}(x,y)\cdot_{\mathcal{A}(T^{\phi^{-1}})}\\\ad_{T^{\phi^{-1}}}(z,w) = \]
\[\ad_{T^{\phi^{-1}}}([x,y,z]_{T^{\phi^{-1}}},w) + \epsilon(x+y,z)\ad_{T^{\phi^{-1}}}(z,[x,y,w])_{T^{\phi^{-1}}}=\]
\[\phi^{-1}(\ad_T([x,y,z]_T^{\phi^{-1}},w) + \epsilon(x+y,z)\ad_T(z,[x,y,w]_T^{\phi^{-1}}))=\]
\[\varphi(\ad_T(x,y)\cdot_{\mathcal{A}(T)}\ad_T(z,w)),\]
\[\varphi(z)\cdot_{\mathcal{A}(T^{\phi^{-1}})}\varphi(w) = z\cdot_{\mathcal{A}(T^{\phi^{-1}})}w = \ad_{T^{\phi^{-1}}}(z,w) = \]
\[\varphi(\ad_T(z,w)) = \varphi(z\cdot_{\mathcal{A}(T)}w),\]
\[\varphi(\ad_T(x,y))\cdot_{\mathcal{A}(T^{\phi^{-1}})}\varphi(w) = \ad_{T^{\phi^{-1}}}(x,y)\cdot_{\mathcal{A}(T^{\phi^{-1}})}w=\]
\[[x,y,w]_{T^{\phi^{-1}}} = [x,y,w]_T^{\phi^{-1}} = \varphi(\ad_T(x,y)\cdot_{\mathcal{A}(T)}w).\]
Completely analogously one can check that that $\varphi$ preserves the products of the type $z\cdot_{\mathcal{A}(T)}\ad_T(x,y).$

Now we check that $\varphi$ commutes with automorphisms. Denote the automorphisms of $\mathcal{A}(T)$ and $\mathcal{A}(T^{\phi^{-1}})$ respectively by $\Phi$ and $\Phi'.$ Then
\[\varphi(\Phi(x)) = \phi(x) = \Phi'(\varphi(x)),\]
\[\varphi(\Phi(\ad_T(x,y))) = \varphi(\ad_T(\phi(x),\phi(y))) = \ad_{T^{\phi^{-1}}}(\phi(x),\phi(y))=\]
\[\Phi'(\ad_{T^{\phi^{-1}}}(x,y)) = \Phi'(\varphi(\ad_T(x,y))).\]
\end{proof}

\medskip

\subsection{Split structures}
In this subsection we introduce the class of split algebras in the framework of regular $Hom$-Leibniz color $3$-algebras and Leibniz color $3$-algebras with automorphism. As in the previous papers on the subject (see \cite{AJ2, AJF2} and other papers on ternary split algebras) this is done by using the standard embedding algebra.

As in the previous section, let $T$ be a regular $Hom$-Leibniz color $3$-algebra (or a Leibniz color $3$-algebra with an automorphism) and $\mathcal{A}$ its standard embedding. Observe that the product in $\mathcal{A}$
gives us a natural action:
\[\mathcal{A}^0 \times \mathcal{A}^1 \to \mathcal{A}^1\]
\[(x,y) \mapsto xy.\]

As we have seen earlier, $\mathcal{A}^0$ is a Lie color algebra. Moreover, the identity (\ref{DI}) shows that that this action endows ${\mathcal A^1}$  with a structure of a color Lie module over ${\mathcal A^0}.$

\begin{definition}
\label{root_spaces}
Let $T = (T,[\cdot,\cdot,\cdot],\epsilon,\phi)$ be a  $Hom$-Leibniz color $3$-algebra (or a Leibniz color 3-algebra with an automorphism) and let ${\mathcal A} = ({\mathfrak{L}}\oplus T,\cdot,\epsilon,\Phi)$ be its standard embedding. Let $H$ be a maximal abelian subalgebra (shortly MASA) of $\mathfrak{L}_0,$ the zeroth $\mathbb G$-component of $\mathfrak{L}.$  The root space of $T$ with respect to $H$ associated to a linear functional $\alpha \in H^{*}$ is the subspace
\[T_{\alpha }:= \{v\in { T}:h\cdot v = \alpha(h) v \, \mbox{ for any } \, h\in H\}.\]
The elements $\alpha \in H^{*}$ such that $T_{\alpha }\neq 0$ are called roots of $T$ (with respect to $H$), and we
 {write} $\Lambda^T:=\{\alpha \in H^{*}\backslash \{0\}: T_{\alpha }\neq 0\}$.
Analogously, by $\Lambda^{{\mathfrak{L}}}$ we denote the set of all nonzero $\alpha \in H^{*}$ such that ${\mathfrak{L}}_{\alpha } \neq 0$, where
\[{\mathfrak{L}}_{\alpha }:=\{e\in {\mathfrak{L}}:[h,e]=\alpha (h)e \, \mbox{ for any } \, h\in H\}\]
are root subspaces of $\mathfrak{L}$ with respect to $T.$
\end{definition}
\begin{remark}Note that despite the fact that we work in the context of $Hom$-algebras, our definition of the root spaces $T_{\alpha}$ does not contain explicitly the homomorphism $\phi$ (cf. \cite{Cao3,Cao4}). This is because we define the root spaces via the multiplication in the algebra $\mathcal{A}(T),$ which is {\bf not} a $Hom$-algebra (but an algebra with an automorphism). In fact, the definition of the product in the standard envelope $\mathcal{A}$ already contains the operator $\phi^{-1},$ which ''cancels out'' with the operator $\phi$ which should be present in the ''usual'' definition of a split $Hom$-algebra (i.e. $\ad(x,y)\cdot w = [x,y,w]^{\phi^{-1}} = \phi^{-1}([x,y,w])$). The material of the next subsection should clarify and justify this idea.
\end{remark}

It is an easy exercise to prove that the root spaces $T_\alpha, \alpha \in \Lambda^T$ and $L_\beta, \beta \in \Lambda^{\mathfrak L}$ are graded (see, for example, \cite{Cao3}). 

The following result collects some basic properties of the
subspaces $T_{\alpha}$ and ${\mathfrak L}_{\alpha}$. The proof is based {on} the Jacobi identity (which holds in  ${\mathfrak{L}}$) and  {on} the identity (\ref{DI}).
{We omit the details here since it is similar to the proof of \cite[Lemma
2.1]{AJF2}}.

\begin{lemma} \label{lema0}
Let  $(T,[\cdot,\cdot,\cdot],\phi,\epsilon)$ be a Leibniz color 3-algebra with automorphism, $\mathcal{A}=\mathfrak L\oplus T$ its standard embedding and $H$ a MASA of ${\mathfrak L}_0$. If $\alpha, \beta ,\gamma \in {\Lambda^T} \cup \{0\}$ and $\delta, \epsilon \in \Lambda^{\mathfrak L} \cup
\{0\}$, the following assertions hold.
\begin{enumerate}
\item[\rm (i)] If $\ad (T_{\alpha }, T_{\beta })\neq 0$
then $\alpha +\beta \in \Lambda^{\mathfrak L} \cup \{0\}$, and
$\ad(T_{\alpha}, T_{\beta})\subset {\mathfrak L}_{\alpha +\beta}.$
\medskip
\item[\rm (ii)]  If ${\mathfrak L}_{\delta } T_{\alpha }\neq 0$ then
$\delta +\alpha \in {\Lambda^T} \cup \{0\}$ and $ {\mathfrak{L}}_{\delta }T_{\alpha }\subset T_{\delta +\alpha }.$
\medskip
\item[\rm (iii)]  If $T_{\alpha} {\mathfrak L}_{\delta} \neq 0$ then $\alpha  + \delta \in {\Lambda^T} \cup \{0\}$ and $T_{\alpha} {\mathfrak L}_{\delta
}\subset T_{\alpha + \delta}.$
\medskip
\item[\rm (iv)] If $[{\mathfrak L}_{\delta } ,{\mathfrak L}_{\epsilon }] \neq 0$ then $\delta
+\epsilon \in \Lambda^{\mathfrak L} \cup \{0\}$ and $ [{\mathfrak
L}_{\delta }, {\mathfrak L}_{\epsilon }] \subset {\mathfrak L}_{\delta +\epsilon }.$
\medskip
\item[\rm (v)] If $[T_{\alpha}, T_{\beta}, T_{\gamma}] \neq 0$ then $\alpha +\beta +\gamma \in
{\Lambda^T} \cup \{0\}$  and $[T_{\alpha}, T_{\beta}, T_{\gamma}] \subset T_{\alpha +\beta +\gamma}.$
\medskip
\item[\rm (vi)] $\phi(T_\alpha) \subseteq T_{\alpha\Phi^{-1}}, \phi^{-1}(T_\alpha) \subseteq T_{\alpha\Phi}$ and $\alpha\Phi^k \in \Lambda^T$ for all $k \in \mathbb{Z}.$
\medskip
\item[\rm (vii)] $\Phi(\mathfrak{L}_\delta) \subseteq \mathfrak{L}_{\delta\Phi^{-1}}, \Phi^{-1}(\mathfrak{L}_\delta) \subseteq \mathfrak{L}_{\delta\Phi}$ and $\delta\Phi^k \in \Lambda^{\mathfrak{L}}$ for all $k \in \mathbb{Z}.$
\end{enumerate}
\end{lemma}

\begin{remark}
Note that the verbatim analogue of the above lemma can be proved for $T$ a regular $Hom$-Leibniz color $3$-algebra, except that in the point (v) we have $[T_{\alpha}, T_{\beta}, T_{\gamma}]^{\phi^{-1}} \subset T_{\alpha +\beta +\gamma}$ instead.
\end{remark}
Now we can give the main definition of the paper.

\begin{definition}
Let $T$ be a $Hom$-Leibniz color $3$-algebra or a Leibniz color 3-algebra with an automorphism and $\mathcal{A} = \mathfrak{L}\oplus T$ its standard embedding. Then $T$ is said to be a split algebra if there exists a MASA $H$ of ${\mathfrak{L}_0}$ such that
\begin{equation} \label{rootdeco}
T = T_{0}\oplus \left (\bigoplus _{\alpha \in \Lambda^T}T_{\alpha
} \right).
\end{equation}
in the sense of the definition \ref{root_spaces}. The set $\Lambda^T$ is called the  root system of $T$. We refer to the decomposition \eqref{rootdeco} as the root spaces decomposition of $T$.
\end{definition}

\begin{remark}\label{r1}
Note that, in contrast to the previous papers considering split ternary algebras, in the color case we cannot guarantee that $[T_0,T_0,T_0] = 0,$ since $\ad(T_0,T_0) \not\subseteq H.$ 
\end{remark}


\subsection{Reduction to the color case}
In this subsection we see that for a split regular $Hom$-Leibniz color $3$-algebra $T$, the algebra $T^{\phi^{-1}}$ is also split, and vice versa. Therefore, we may consider only the case of Leibniz color $3$-algebras with an automorphism.

Let $(T,[\cdot,\cdot,\cdot],\phi,\epsilon)$ be a regular $Hom$-Leibniz color $3$-algebra, and let $H$ be a MASA of $\mathfrak{L}, \alpha \in \Lambda, v \in T_{\alpha}.$ Consider a Leibniz color $3$-algebra $T^{\phi^{-1}}.$ Recall that the algebras $\mathcal{A}(T)$ and $\mathcal{A}(T^{\phi^{-1}})$ are isomorphic (via the map $\varphi$) as 2-graded color algebras with an automorphism. Hence, the subalgebra $\varphi(H) = \phi^{-1}H$ is a MASA of $\mathfrak{L}(T^{\phi^{-1}}).$ Now, $h = \varphi(h') \in \varphi(H)$ and since $\varphi$ is an algebra isomorphism, we get
\[h\cdot_{\mathcal{A}(T^{\phi^{-1}})}v = \varphi(h')\cdot_{\mathcal{A}(T^{\phi^{-1}})}\varphi(v) = \varphi(h'\cdot_{\mathcal{A}(T)}v) = \varphi(\alpha(h')v)=\alpha(h')v = \alpha(\varphi^{-1}(h))v = \alpha(\phi h)v.\]
This clearly implies that
\begin{equation}
\label{dehom_root_spaces}
T_{\alpha} = (T^{\phi^{-1}})_{\alpha\circ\varphi^{-1}},
\end{equation}
and if $T$ is a split  $Hom$ Leibniz color 3-algebra (\ref{rootdeco}), then we have
\[T^{\phi^{-1}} = T_{0}\oplus \big (\bigoplus _{\alpha \in \Lambda^T}T_{\alpha\circ\varphi^{-1}} \big).\]
Therefore, $T^{\phi^{-1}}$ is a split Leibniz color $3$-algebra with the root system $\Lambda^{T^{\phi^{-1}}} = \Lambda^T\circ\varphi^{-1}.$ Clearly, the converse statement also holds.

Hence, we may reduce our considerations and obtain our results only in the case of split  Leibniz color $3$-algebras. Having obtained our main results for the algebra $T^{\phi^{-1}}$, we can go back to the $Hom$ case by reversing the process above and obtain the analogous results. This will be done explicitly in the subsequent sections.

\section{Split Leibniz color $3$-algebras}

\subsection{The spaces associated to roots}

In this supplementary subsection for any roots $\alpha, \beta \in \Lambda^T$ we introduce certain subspaces $T_{0,\alpha}$ and $T(\alpha,\beta)$ which are going to be important later in the discussion.

Let $\alpha \in \Lambda^T$ be a root. We construct the space $T_{0,\alpha}$ inductively: deifne $T_{0,\alpha}^{(0)} = [T_{\alpha},T_{-\alpha},T_0],$ and for $k \geq 1$ define $T_{0,\alpha}^{(k)} = [T_{0,\alpha}^{(k-1)},T_0,T_0].$ Finally, we define \begin{equation}
\label{t0alpha}
T_{0,\alpha} = \sum_{k\geq 0} T_{0,\alpha}^{(k)}.
\end{equation}

The spaces $T_{0,\alpha}^{(k)}$ satisfy some relations:

\begin{lemma}
\label{t0alpha_prop}
$1) [T_0,T_0,T_{0,\alpha}^{(k)}] \subseteq T_{0,\alpha}^{(k)},$\\
$2) [T_0,T_{0,\alpha}^{(k)},T_0] \subseteq T_{0,\alpha}^{(k)} + T_{0,\alpha}^{(k+1)}.$
\end{lemma}
The proof of the lemma is a simple induction applying the Leibniz identity.

Now, let $\alpha, \beta \in \Lambda^T$ be any two roots. Consider the following subspace:
\[T(\alpha,\beta) = \Big([T_\alpha,T_{-\alpha},T_\beta]\Big)_{\circlearrowleft \alpha\leftrightarrow -\alpha, \beta \leftrightarrow - \beta, \alpha \leftrightarrow \beta} + \Big( [T_{0,\alpha},T_\beta,T_0] + [T_{0,\alpha},T_{0,\beta},T_0] \Big)_{\circlearrowleft S^3,\alpha\leftrightarrow -\alpha, \beta \leftrightarrow - \beta, \alpha \leftrightarrow \beta},\]
where the below index means that we also add to the sum all spaces of the following form, changing the order in the triple product or interchanging $\alpha$ and $-\alpha,$ $\beta$ and $-\beta$ and $\alpha$ and $\beta.$ Hence, by construction, this space is symmetrical with respect to the symmetries above:
\[T(\alpha,\beta) = T(-\alpha,\beta) = T(\alpha,-\beta) = T(\beta,\alpha) = \ldots\]  
\subsection{Connections of roots}
Our main tool is the so-called connections of roots. In this section we give the notion of connectivity of roots for Leibniz color $3$-algebras with automorphism. In what follows, $T$ denotes a split Leibniz
color $3$-algebra and $T = T_{0}\oplus \big(\bigoplus_{\alpha \in
\Lambda^T}T_{\alpha }\big)$  {is} its corresponding root spaces
decomposition.

\begin{definition}
The root system $\Lambda^T$ ($\Lambda^{\mathfrak L}$) is called symmetric if $\Lambda^T = -\Lambda^T$ ($\Lambda^{\mathfrak L} = -\Lambda^{\mathfrak L}$), where for $\emptyset \neq \Upsilon \subset H^*$ the set  $-\Upsilon$ is just $\{-\alpha: \alpha \in \Upsilon\}$.
\end{definition}
From now on, we will suppose that both root systems $\Lambda^T, \Lambda^{\mathfrak L}$ are symmetric.

\smallskip



For each $\alpha\in\Lambda^T$ introduce a new symbol $\theta_{\alpha}$, and let $\Theta = \{ \theta_{\alpha}: \alpha \in \Lambda^T \}.$ Define the new operation:
\[
\dotplus: (\Lambda^T \cup  \Lambda^{\mathfrak
L} \cup \Theta_{\Omega}) \times ( \Lambda^T \cup \{0\}) \to H^*
\cup \Theta_{\Omega},
\]
 {defined as follows:}
\begin{itemize}

\item For $\alpha \in  \Lambda^T \cup  \Lambda^{\mathfrak{L}}$ and $\beta \in  \Lambda^T
\cup \{0\}$ with $\beta \neq -\alpha$, the new operation $\alpha \dotplus \beta \in H^*$ coincides with the
usual sum of linear functionals.
\medskip
\medskip
\item For $\alpha \in  \Lambda^T$,
\[
\alpha \dotplus (-\alpha) = \theta_\alpha.
\]

\medskip
\item For $\theta_{\alpha} \in \Theta$ and $\beta \in 
\Lambda^T$,
\[
\quad \quad \quad \quad
\theta_{\alpha} \dotplus \beta = \left \{
\begin{array}{llll}
\beta, & \mbox{ if } \sum_{k \in \mathbb{Z}}T(\alpha\Phi^k,\beta\Phi^k) \neq 0 \\
0, & \mbox{ otherwise.} \\
\end{array}
\right.
\]
\medskip
\item For $\theta_{\alpha}\in \Theta$ and $0:H \to {\hu
K}$, we define $\theta_{\alpha} \dotplus 0 = 0$.
\end{itemize}


\begin{remark}\label{remark1}
Note that $\alpha \dotplus \beta$ {coincides} with $\alpha + \beta$ for $\alpha \in \Lambda^T \cup 
\Lambda^{\mathfrak L}$ and $\beta \in \Lambda^T \cup \{0\}$,  {except for the case}
$\beta = -\alpha$, where $\alpha \dotplus -\alpha = \theta_\alpha.$
\end{remark}

The following results are the basic properties of the operation $\dotplus$ which we will use later (see also \cite[Lemma 3.1]{AJF2} and
\cite[Lemmas 3.2 and 3.3]{AJF2}).

\begin{lemma}\label{lemao}
For any $\alpha, \beta \in \Lambda^T$ such
that $\theta_{\alpha} \dotplus \beta = \beta$
the following assertions hold.

\begin{enumerate}
\item[{\rm (i)}]  $\theta_{\beta} \dotplus \alpha = \alpha$.
\medskip
\item[{\rm (ii)}]  $\theta_{-\alpha} \dotplus \beta = \beta$.
\medskip
\item[{\rm (iii)}]  $\theta_{-\alpha} \dotplus (-\beta) =
-\beta$.
\medskip
\item[{\rm (iv)}] $\theta_{\alpha\Phi^k}\dotplus\beta\Phi^k = \beta\Phi^k$ for all $k \in \mathbb{Z}.$
\end{enumerate}
\end{lemma}
\begin{proof}
Just use the symmetry of the space $T(\alpha,\beta)$ with respect to the symmetries $\alpha \leftrightarrow -\alpha, \beta \leftrightarrow -\beta, \alpha \leftrightarrow \beta$ and the definition of the operation $\dotplus.$ 
\end{proof}


\begin{lemma}
\label{lemasym}
Let $\alpha, \delta \in \Lambda^T, \beta, \gamma \in \Lambda^T \cup \{0\},$ and let $\alpha\dotplus\beta \in \Lambda^{\mathfrak{L}}\cup\Theta, (\alpha\dotplus\beta)\dotplus\gamma = \delta.$ Then the following assertions hold:
\begin{enumerate}
\item[{\rm (i)}] $\alpha\Phi^k\dotplus\beta\Phi^k \in \Lambda^{\mathfrak{L}}\cup\Theta, (\alpha\Phi^k\dotplus\beta\Phi^k)\dotplus\gamma\Phi^k = \delta\Phi^k$ for any $k \in \mathbb{Z}.$
\medskip
\item[{\rm (ii)}]  $(-\alpha)\dotplus(-\beta) \in \Lambda^{\mathfrak{L}}\cup\Theta, ((-\alpha)\dotplus(-\beta))\dotplus(-\gamma) = -\delta.$
\medskip
\item[{\rm (iii)}] $\delta \dotplus (-\gamma) \in \Lambda^{\mathfrak{L}}\cup\Theta, (\delta \dotplus (-\gamma)) \dotplus (-\beta) = \alpha.$
\end{enumerate}
\end{lemma}
\begin{proof}
(i) Suppose first that $\alpha\dotplus\beta \in \Lambda^{\mathfrak{L}}.$ Then by remark \ref{remark1}, we have $\alpha\dotplus\beta = \alpha + \beta, (\alpha\dotplus\beta)\dotplus\gamma = \alpha + \beta + \gamma = \delta.$ Multiplying these equations by $\Phi^k,$ we get $\alpha\Phi^k \dotplus \beta\Phi^k = \alpha\Phi^k + \beta\Phi^k = (\alpha+\beta)\Phi^k \in \Lambda^{\mathfrak{L}}, (\alpha\Phi^k\dotplus\beta\Phi^k)\dotplus\gamma\Phi^k = \alpha\Phi^k + \beta\Phi^k + \gamma\Phi^k = \delta\Phi^k.$

Now suppose that $\alpha + \beta = \theta_\alpha$ and $\theta_\alpha + \gamma = \delta.$ Then $\alpha + \beta = 0, \gamma = \delta$ and $\alpha\Phi^k \dotplus \beta\Phi^k = \theta_{\alpha\Phi^k}.$ Since $\theta_\alpha \dotplus \gamma = \gamma,$ 
the point (iv) of the above lemma implies that
$\theta_{\alpha\Phi^k} \dotplus \gamma\Phi^k = \gamma\Phi^k = \delta\Phi^k$ for all $k \in \mathbb{Z}.$ 

\medskip

The point (ii) is proved completely analogously, just use the
point (iii) of the previous lemma.

\medskip

(iii) Suppose first that $\alpha\dotplus\beta \in \Lambda^{\mathfrak{L}}.$ Then again we have $\alpha\dotplus\beta = \alpha + \beta, (\alpha\dotplus\beta)\dotplus\gamma = \alpha + \beta + \gamma = \delta.$ Therefore, $\delta \dotplus (-\gamma) = \delta - \gamma = \alpha + \beta \in \Lambda^{\mathfrak{L}}, (\delta \dotplus (-\gamma)) \dotplus (-\beta) = \delta - \gamma - \beta = \alpha.$

Now suppose that $\alpha \dotplus \beta \in \Theta.$ Then again $\alpha + \beta = 0, \gamma = \delta$ and $\theta_\alpha \dotplus \gamma = \gamma.$ Hence, $\delta \dotplus (-\gamma) = \theta_\delta \in \Theta$ and $(\delta \dotplus (-\gamma)) \dotplus (-\beta) = \theta_\gamma \dotplus \alpha = \alpha$
by the point (i) of the previous lemma.
\end{proof}

We are now ready to introduce the key tool in our study.

\begin{definition}\label{defco}
Let $\alpha$ and $\beta$ be two nonzero roots in $\Lambda^T.$ We say that
$\alpha$ is connected to $\beta$ (and denote it by $\alpha \sim \beta$) if there exists an ordered set
$\{\alpha_1,\alpha_2,\ldots,\alpha_{2n},\alpha _{2n+1}\}\subset
\Lambda^T \cup \{0\}$  satisfying the following conditions: 
\begin{itemize}
\item[{\bf 1.}] $\alpha_1=\alpha.$
\medskip
\item[{\bf 2.}] An odd numbers of factors operated under $\dotplus$ belongs to $\pm \Lambda^T$. More
precisely,
\begin{align*}
\big \{\alpha_1, & (\alpha_1 \dotplus
\alpha_2) \dotplus \alpha_3,
 (((\alpha_1 \dotplus \alpha_2) \dotplus \alpha_3) \dotplus \alpha_4) \dotplus \alpha_5, \ldots,
\\
& ((\ldots((\alpha_1 \dotplus \alpha_2) \dotplus \alpha_3)  \dotplus \ldots) \dotplus \alpha _{2n}) \dotplus \alpha_{2n+1} \big \} \subset \pm
\Lambda^T
\end{align*}
\item [{\bf 3.}] The result of the operation of an even numbers of
factors under $\dotplus$ belongs to
$\pm \Lambda^{\mathfrak L}$ or $\Theta$, that
is,
\begin{align*}
\big \{
(\alpha_1  \dotplus \alpha_2 ), ((\alpha_1  \dotplus \alpha_2 ) \dotplus \alpha_3) \dotplus \alpha_4 , \ldots,
\\
& ((\ldots((\alpha_1  \dotplus \alpha _{2} ) \dotplus \alpha_3 )  \dotplus
\ldots) \dotplus \alpha_{2n-1} ) \dotplus \alpha_{2n}  \big \} \subset \pm
\Lambda^{\mathfrak L}\cup \Theta.
\end{align*}
\item [{\bf 4.}]
$
((\ldots((\alpha_1 \dotplus \alpha_2) \dotplus \alpha_3) \dotplus
\ldots) \dotplus \alpha_{2n}) \dotplus \alpha _{2n+1} = \pm \beta\Phi^k, k \in \mathbb{Z}.
$
\end{itemize}
The family $\{\alpha_1,\alpha_2,\ldots \alpha_{2n},\alpha _{2n+1}\}$ is called a  connection from $\alpha$ to $\beta$.
\end{definition}

By definition, for any $\alpha \in \Lambda^{T}$ the set
$\{\alpha\}$ is a connection from $\alpha$ to any root $\pm\alpha\Phi^k, k \in \mathbb{Z}.$ In particular, the relation $\sim$ is reflexive.

Let $\{\alpha_{1},\alpha _{2},\alpha _{3}, \ldots ,\alpha_{2n-1},\alpha_{2n}, \alpha_{2n+1}\}\subset \pm \Lambda^T \cup \{0\}$ be a connection from $\alpha$ to $\beta$ such that
\[\hbox{$((\ldots((\alpha_{1}\dotplus \, \alpha_{2})\dotplus
\alpha_3)\dotplus \ldots )\dotplus\alpha_{2n})\dotplus \alpha _{2n+1} =
\varepsilon \beta\Phi^k,$ where $\varepsilon \in \{\pm1\}, k \in \mathbb{Z}$.}\]
If $n = 0$, then $\alpha = \pm\beta\Phi^{-k}$. So $\{\beta\}$ is a connection from $\beta$ to $\alpha$.
If $n\geq 1$, then points (i), (ii) of lemma \ref{lemasym} imply that
\[\hbox{$((\ldots((\varepsilon\alpha_{1}\Phi^{-k}\dotplus \, \varepsilon\alpha_{2}\Phi^{-k})\dotplus
\varepsilon\alpha_{3}\Phi^{-k})\dotplus \ldots )\dotplus\varepsilon\alpha_{2n}\Phi^{-k})\dotplus \varepsilon\alpha_{2n+1}\Phi^{-k} =
\beta,$}\]
and (iii) of lemma \ref{lemasym} implies that 
$\{\beta ,-\varepsilon\alpha_{2n+1}\Phi^{-k} ,-\varepsilon\alpha_{2n}\Phi^{-k},\ldots, -\varepsilon\alpha_3\Phi^{-k},-\varepsilon\alpha_2\Phi^{-k}\}$ is a connection from $\beta$ to $\alpha$. That is, $\sim$ is symmetric.

Now, let $\{\alpha _{1},\alpha _{2},\ldots, \alpha_{2n+1}\}$ and $\{\beta
_{1},\beta _{2}, \ldots , \beta_{2m+1}\}$ be connections from $\alpha$ to $\beta$ and from $\beta$
to $\gamma$, respectively, such that 
\[(\ldots(\alpha _{1}\dotplus \alpha_{2})\dotplus \ldots)\dotplus
\alpha_{2n+1}=\varepsilon_1\beta\Phi^k,\]
\[(\ldots(\beta_{1}\dotplus \beta_{2})\dotplus \ldots)\dotplus
\beta_{2m+1}=\varepsilon_2\gamma\Phi^l,\]
where $\varepsilon_1, \varepsilon_2 = \pm 1, k, l \in \mathbb{Z}.$ Again, points (i), (ii) of lemma \ref{lemasym} imply that 
\[(\ldots(\varepsilon_1\beta_1\Phi^k\dotplus \varepsilon_1\beta_2\Phi^k)\dotplus \ldots)\dotplus
\varepsilon_1\beta_{2m+1}\Phi^k=\varepsilon_1\varepsilon_2\gamma\Phi^{k+l}.\]
Thus, if $m=0$, then $\beta = \varepsilon_2\gamma\Phi^l$ and $\{\alpha
_{1},\alpha _{2}, \ldots , \alpha_{2n+1}\}$ is a connection from
$\alpha$ to $\gamma$, and if $m\geq 1$,  
$\{\alpha _{1},\alpha _{2},\ldots, \alpha_{2n+1},\varepsilon_1 \beta
_{2}\Phi^k,\ldots, \epsilon_1 \beta_{2m+1}\Phi^k\}$ is a connection from $\alpha$
to $\gamma$. That is, $\sim$ is transitive.  We have proved the following statement:



\begin{proposition}\label{equivalence}
Suppose that the root systems $\Lambda^T, \Lambda^{\mathfrak L}$ are symmetric. Then the relation $\sim$ on $\Lambda^T$ is an equivalence
relation.
\end{proposition}

\subsection{Decompositions}
Let $T$ be a split Leibniz color $3$-algebra with an automorphism and $T = T_{0}\oplus \big(\bigoplus_{\alpha \in \Lambda^T}T_{\alpha }\big)$
its corresponding root spaces decomposition. By Proposition \ref{equivalence} the connectivity relation is an equivalence
relation in $\Lambda^T,$ so we can consider the quotient set $$\Lambda^T / \sim=\{[\alpha]: \alpha \in \Lambda^T\},$$
 {where} $[\alpha]$  {is} the set of nonzero roots of $T$ which are connected to $\alpha$. By the definition of $\sim$, it is clear that if $\beta \in [\alpha]$  then $\pm\beta\Phi^k \in [\alpha]$ for all $k \in \mathbb{Z}$.

We also list some important particular cases in which two roots are connected:

\begin{lemma}
\label{lema1}
Let $\alpha, \beta \in \Lambda^T$ be roots.
\begin{enumerate}
\item[{\rm (i)}] If
$\alpha + \beta  \in \Lambda^{\mathfrak L}\cup \{0\}$, then $\alpha \sim \beta.$
\item[{\rm (ii)}] If $T(\alpha,\beta) \neq 0,$ then $\alpha \sim \beta.$
\end{enumerate}
\end{lemma}
\begin{proof}
(i) One can easily check that the set $\{\alpha,\beta,0\}$ is a connection from $\alpha$ to $\beta.$

(ii) One can easily check that the set $\{\alpha,-\alpha,\beta\}$ is a connection from $\alpha$ to $\beta.$
\end{proof}

\smallskip

Our goal in this section is to associate an ideal $T_{[\alpha]}$  {of $T$}
to  {each} $[\alpha]$.  {Given} $\alpha \in \Lambda^T$, we start by
defining the set $T_{0,[\alpha]} \subset T_0$ as follows:
\begin{equation}
\label{t0[alpha]}
T_{0,[\alpha]} = {\rm span}_{\hu K}\big \{\,[T_{\beta},T_{\gamma},
T_{\delta}]: \beta, \gamma, \delta \in [\alpha] \cup \{0\}, \beta + \gamma + \delta = 0 \mbox{ and } \delta \neq 0  \big \} + \sum_{\beta \in [\alpha]} T_{0,\beta}.    
\end{equation}
Next, we define
\begin{equation}
\label{v[alpha]}
V_{[\alpha]}:=\bigoplus_{\beta \in [\alpha]}T_{\beta}
\end{equation}  and
\begin{equation}
\label{t[alpha]}
T_{[\alpha]}:= T_{0,[\alpha]}\oplus {V}_{[\alpha]}.
\end{equation}

\begin{remark}
Note that the space $T_{0,[\alpha]}$ collapses in special cases. For example, if the $\mathbb{G}$-grading is trivial, then for any $\alpha \in \Lambda^T$ we have $[T_\alpha,T_{-\alpha},T_0] \subseteq HT_0 = 0.$ If $T$ is a $Hom$-Lie color 3-algebra, then $T_{0,\alpha}^{(k)} = [T_{0,\alpha}^{(k-1)},T_0,T_0] = [T_0,T_0,T_{0,\alpha}^{(k-1)}] \subseteq T_{0,\alpha}^{(k-1)},$ and for any $\beta \in [\alpha]$ we have $[T_\beta,T_{-\beta},T_0] = [0,T_\beta,T_{-\beta}]$. Hence, in this case $T_{0,[\alpha]} = {\rm span}_{\hu K}\big \{\,[T_{\beta},T_{\gamma},
T_{\delta}]: \beta, \gamma, \delta \in [\alpha] \cup \{0\}, \beta + \gamma + \delta = 0 \mbox{ and } \delta \neq 0  \big \}.$ In fact, the same result holds for $T$ a regular $Hom$-Lie triple system (for more details, see lemma \ref{lema51}). 
\end{remark}



Our aim is to show that $T_{[\alpha]}$ is an ideal of $T$ for any $[\alpha]$. We begin by considering the products involving the spaces $T_\beta, \beta \in [\alpha].$


\begin{lemma}\label{lema3}
Fix $\alpha_0 \in \Lambda^T$. The following assertions hold for
$\alpha \in [{\alpha_0}]$ and $\beta, \gamma \in \Lambda^T \cup
\{0\}$.

\begin{enumerate}
\item[{\rm (i)}] If $\big[T_\alpha, T_\beta, T_\gamma \big] \neq
0$, then $\beta, \gamma, \alpha + \beta + \gamma \in [{\alpha_0}]
\cup \{0\}$ and $\big[T_\alpha, T_\beta, T_\gamma \big] \subseteq T_{[\alpha_0]}$.
\smallskip
\item[{\rm (ii)}] If $\big[T_\beta, T_\alpha, T_\gamma \big] \neq 0$, then
 $\beta, \gamma, \beta +
\alpha  + \gamma \in [{\alpha_0}] \cup \{0\}$ and $\big[T_\beta, T_\alpha, T_\gamma \big] \subseteq T_{[\alpha_0]}$.
\smallskip
 \item[{\rm (iii)}] If $\big[T_\beta, T_\gamma, T_\alpha \big] \neq 0$,
then $\beta, \gamma, \beta + \gamma  + \alpha \in [{\alpha_0}]
\cup \{0\}$ and $\big[T_\beta, T_\gamma, T_\alpha \big] \subseteq T_{[\alpha_0]}$.
\end{enumerate}
\end{lemma}

\begin{proof}
This lemma is proved analogously to the lemma \cite[Lemma 4.4]{3L}, so we only consider the case (i) (to gain insight into the definition of the operation $\dotplus$ and the spaces $T(\alpha,\beta)$) and direct the interested reader to the paper mentioned above for other cases.

(i) By (i) of lemma \ref{lema0} $\alpha + \beta \in \Lambda^{\mathfrak{L}} \cup \{0\},$ hence, $\beta \in  [\alpha_0] \cup \{0\}$ by (i) of lemma \ref{lema1}. Consider
two cases:
\begin{enumerate}
\item $\alpha + \beta + \gamma = 0.$

It remains to check that $\gamma \in [{\alpha_0}] \cup \{0\}$. If
$\gamma \neq 0$ since $-\gamma = \alpha + \beta \in \Lambda^{\mathfrak{L}}$, one can see that the family $\{\alpha, \beta,0 \}$ is a connection from $\alpha$ to $\gamma$.  
\smallskip
\item $\alpha + \beta + \gamma \neq 0.$ Moreover, $[T_\alpha,T_\beta,T_\gamma] \subseteq T_{0,[\alpha_0]}$ by the definition,

Assume first that $\alpha + \beta \neq 0$, then $\alpha + \beta
\in \Lambda^{\mathfrak{L}}$ and $\{\alpha, \beta, \gamma \}$ is a connection from $\alpha$ to $\alpha + \beta + \gamma$ and the whole product lies in $T_{\alpha+\beta+\gamma} \subseteq V_{[\alpha_0]}.$

If $\gamma \neq 0$, then the family $\{\alpha,
\beta, -\alpha- \beta -\gamma \}$ is a connection from $\alpha$ to $\gamma$.

Finally, if $\alpha + \beta = 0$ then necessarily $\gamma \neq 0$ and $[T_\alpha,T_{-\alpha},T_\gamma] \neq 0.$ Therefore, $T(\alpha,\gamma) \neq 0$ and $\alpha \sim \gamma$ by (ii) of lemma \ref{lema1}.
\end{enumerate}
\end{proof}

In the next two lemmas we consider the products involving the space $T_{0,[\alpha_0]}.$

\begin{lemma}\label{lema51}
Fix $\alpha_0 \in \Lambda^T$. Let $\alpha, \beta  \in
[\alpha_0] \cup \{0\}, \gamma \in [\alpha_0]$ with $\alpha + \beta + \gamma = 0$ and
$\delta, \epsilon \in \Lambda^T \cup \{0\}$. Then the following
assertions hold:

\begin{enumerate}
\item[\rm(i)] If $[[T_\alpha, T_\beta,
T_\gamma],T_\delta,T_\epsilon] \neq 0 $ then $\delta, \epsilon,
\delta + \epsilon \in [{\alpha_0}]\cup \{0\}$ and $[[T_\alpha, T_\beta,
T_\gamma],T_\delta,T_\epsilon] \subseteq T_{[\alpha_0]}$.
\smallskip
\item[\rm(ii)] If  $[T_\delta,  [T_\alpha, T_\beta, T_\gamma],
T_{\epsilon}\big] \neq 0$ then $\delta, \epsilon, \delta  +
\epsilon \in [{\alpha_0}] \cup \{0\}$ and $[T_\delta,  [T_\alpha, T_\beta, T_\gamma],
T_{\epsilon}] \subseteq T_{[\alpha_0]}$.
\smallskip
\item[\rm (iii)] If  $[T_\delta,T_\epsilon,[T_\alpha, T_\beta,
T_\gamma]] \neq 0$ then $\delta, \epsilon, \delta + \epsilon
\in [{\alpha_0}]\cup \{0\}$ and $[T_\delta,T_\epsilon,[T_\alpha, T_\beta,
T_\gamma]] \subseteq T_{[\alpha_0]}$.
\end{enumerate}
\end{lemma}

\begin{proof}
We only prove the part (i), the parts (ii) and (iii) are proved analogously.

By the Leibniz identity, we have 
\[[[T_\alpha, T_\beta,
T_\gamma],T_\delta,T_\epsilon] \subseteq [T_\alpha, T_\beta, [T_\gamma, T_\delta, T_\epsilon]] + [T_\gamma, [T_\alpha, T_\beta, T_\delta],T_\epsilon] + [T_\gamma, T_\delta, [T_\alpha, T_\beta, T_\epsilon]],\]
so at least one of the spaces on the right side is nonzero. Consider three cases:

\begin{enumerate}
\item[\rm(1)] Suppose that $[T_\alpha, T_\beta, [T_\gamma, T_\delta, T_\epsilon]] \neq 0.$ Since  $[T_\gamma, T_\delta, T_\epsilon] \neq 0$ and $\gamma \neq 0$, by the previous lemma $\delta, \epsilon \in [\alpha_0] \cup \{0\}$. By the lemma \ref{lema0}, $[T_\alpha, T_\beta, [T_\gamma, T_\delta, T_\epsilon]] \subseteq [T_\alpha, T_\beta, T_{\gamma+\delta+\epsilon}].$ Since at least one of the $\alpha, \beta$ is nonzero, we may apply the previous lemma once more to get $\alpha + \beta + \gamma + \delta + \epsilon = \delta + \epsilon \in [\alpha_0] \cup \{0\}.$

Moreover, if $\delta + \epsilon \neq 0,$ then the whole product lies in $T_{\delta + \epsilon} \subseteq V_{[\alpha_0]},$ and if $\delta + \epsilon = 0,$ then the whole product lies in $[T_\alpha,T_\beta,T_\gamma] \subseteq T_{0,[\alpha_0]}.$

\item[\rm(2)] Now suppose that $[T_\gamma, [T_\alpha, T_\beta, T_\delta],T_\epsilon] \neq 0.$ Since $[T_\alpha, T_\beta, T_\delta]$ and at least one of $\alpha, \beta$ is nonzero, the previous lemma implies that $\delta, \alpha + \beta + \delta \in [\alpha_0] \cup \{0\}.$ By the lemma \ref{lema0}, $[T_\gamma, [T_\alpha, T_\beta, T_\delta],T_\epsilon] \subseteq [T_\gamma, T_{\alpha + \beta +\delta}, T_\epsilon].$ Since $\gamma \neq 0,$ the previous lemma implies that $\epsilon \in [\alpha_0] \cup \{0\}, \alpha + \beta + \gamma + \delta + \epsilon = \delta + \epsilon \in [\alpha_0] \cup \{0\}.$

Moreover, if $\delta + \epsilon \neq 0,$ then the whole product lies in $T_{\delta + \epsilon} \subseteq V_{[\alpha_0]}.$ If $\delta = -\epsilon,$ then the whole product lies in $[T_\gamma, T_{\alpha + \beta + \delta}, T_{-\delta}] \subseteq T_{0,[\alpha_0]}$ since all roots indexing the spaces in the product are in $[\alpha_0]\cup\{0\}$ and $\gamma \neq 0.$

\item[\rm(3)] The last case in which $[T_\gamma, T_\delta, [T_\alpha, T_\beta, T_\epsilon]] \neq 0$ is completely analogous to the case (2).
\end{enumerate}
\end{proof}

\medskip

\begin{lemma}
\label{lema52}
Let $\alpha \in [\alpha_0], \delta, \epsilon \in \Lambda^T \cup \{0\},$ and $k \in \mathbb{N}.$ Then the following assertions hold:

\begin{enumerate}
\item[\rm(i)] If $[T_{0,\alpha}^{(k)},T_\delta,T_\epsilon] \neq 0 $ then $\delta, \epsilon,
\delta + \epsilon \in [{\alpha_0}]\cup \{0\}$ and $[T_{0,\alpha}^{(k)},T_\delta,T_\epsilon] \subseteq T_{[\alpha_0]}$.
\smallskip
\item[\rm(ii)] If  $[T_\delta,T_{0,\alpha}^{(k)},
T_{\epsilon}] \neq 0$ then $\delta, \epsilon, \delta  +
\epsilon \in [{\alpha_0}] \cup \{0\}$ and $[T_\delta,T_{0,\alpha}^{(k)},
T_{\epsilon}] \subseteq T_{[\alpha_0]}$.
\smallskip
\item[\rm (iii)] If  $[T_\delta,T_\epsilon,T_{0,\alpha}^{(k)}] \neq 0$ then $\delta, \epsilon, \delta + \epsilon
\in [{\alpha_0}]\cup \{0\}$ and $[T_\delta,T_\epsilon,T_{0,\alpha}^{(k)}] \subseteq T_{[\alpha_0]}$.
\end{enumerate}

\end{lemma}
\begin{proof}
We use induction on $k$ and begin with $k = 0.$ Consider, for example the point (i): let $\alpha \in [\alpha_0]$ and let $[[T_{\alpha},T_{-\alpha},T_0],T_{\delta},T_{\epsilon}] \neq 0.$ As in the previous lemma, apply the  Leibniz identity to this product:
\[[[T_{\alpha},T_{-\alpha},T_0],T_{\delta},T_{\epsilon}] \subseteq [T_\alpha, T_{-\alpha}, [T_0, T_\delta, T_\epsilon]] + [T_0,[T_\alpha,T_{-\alpha},T_\delta],T_\epsilon] + [T_0,T_\delta,[T_\alpha,T_{-\alpha},T_\epsilon]]\]
and consider three cases:
\begin{enumerate}
    \item[\rm(1)] Suppose that $[T_\alpha, T_{-\alpha}, [T_0, T_\delta, T_\epsilon]] \neq 0.$ By lemmas \ref{lema0}, \ref{lema3} we have 
\[[T_\alpha, T_{-\alpha}, [T_0, T_\delta, T_\epsilon]] \subseteq [T_{\alpha},T_{-\alpha},T_{\delta+\epsilon}] \subseteq T_{\delta + \epsilon}\]
and $\delta + \epsilon \in [\alpha_0]\cup\{0\}.$ If any of $\delta, \epsilon$ is zero, then we are done. Suppose then that both $\delta, \epsilon \neq 0.$ Since $[T_0,T_\delta,T_\epsilon] \neq 0,$ the lemma \ref{lema3} implies that $\delta \sim \epsilon.$ If $\delta + \epsilon \neq 0,$ then $\delta \sim \delta + \epsilon \sim \alpha_0.$ If $\delta = -\epsilon,$ then we have
\[0 \neq [[T_\alpha, T_{-\alpha}, [T_0, T_\delta, T_{-\delta}]],\]
which implies $\alpha \sim \delta$ by (iii) of the previous lemma and we are done.

Moreover, if $\delta + \epsilon \neq 0,$ then the whole product lies in $T_{\delta + \epsilon} \subseteq V_{[\alpha_0]},$ and if $\delta + \epsilon = 0,$ then the whole product lies in $[T_\alpha,T_{-\alpha},T_0] \subseteq T_{0,[\alpha_0]}.$

\item[\rm(2)] Suppose now that $[T_0,[T_\alpha,T_{-\alpha},T_\delta],T_\epsilon] \neq 0.$ This implies that $[T_\alpha,T_{-\alpha},T_\delta] \neq 0$ and by the lemma \ref{lema3} $\delta \in [\alpha_0] \cup \{0\}.$ By the lemma \ref{lema0}, $[T_0,T_\delta,T_\epsilon] \neq 0.$ If $\delta \neq 0,$ therefore the lemma \ref{lema3} implies that $\epsilon, \delta + \epsilon \in [\delta]\cup\{0\} = [\alpha_0] \cup \{0\}$ and we are done. If $\delta = 0, \epsilon \neq 0,$ then the condition $[T_0,[T_\alpha,T_{-\alpha},T_0],T_\epsilon] \neq 0$ implies that $T(\alpha,\epsilon) \neq 0$ and $\alpha \sim \epsilon$ by lemma \ref{lema1}. If $\delta = \epsilon = 0,$ then there is nothing to prove.

If at least one of $\delta, \epsilon$ is nonzero, then the whole product lies in $[T_0,T_\delta,T_\epsilon] \subseteq T_{\delta +\epsilon} \subseteq V_{[\alpha_0]}$ if $\delta + \epsilon \neq 0$ or in $T_{0,[\alpha_0]}$ if $\delta + \epsilon = 0.$  If $\delta = \epsilon = 0$ then the product lies in $[T_0,[T_\alpha, T_{-\alpha},T_0],T_0] \subseteq T_{0,\alpha}^{(0)} + T_{0,\alpha}^{(1)} \subseteq T_{0,[\alpha_0]}$ by the lemma \ref{t0alpha_prop}.

\item[\rm(3)] The last case, where $[T_0,T_\delta,[T_\alpha,T_{-\alpha},T_\epsilon]] \neq 0$ is completely analogous to the case (2).
\end{enumerate}

\medskip

Now suppose that the lemma holds for all $i \leq k$ and consider it for $k+1.$ Consider, for example, the case (iii): let $[T_\delta, T_\epsilon, T_{0,\alpha}^{(k+1)}] \neq 0.$ Writing $T_{0,\alpha}^{(k+1)}$ as $[T_{0,\alpha}^{(k)},T_0,T_0]$ and applying the Leibniz identity to the product, we get
\[[T_\delta, T_\epsilon, T_{0,\alpha}^{(k+1)}] \subseteq [T_{0,\alpha}^{(k)},T_0,[T_\delta,T_\epsilon,T_0]] + [[T_{0,\alpha}^{(k)},T_0,T_\delta],T_\epsilon,T_0] + [T_\delta,[T_{0,\alpha}^{(k)},T_0,T_\epsilon],T_0].\]
Therefore, one of the spaces on the right is nonzero. We can analyze the cases separately, supposing that each of the spaces is nonzero. This analysis is very similar to the one at the base of induction.

Suppose, for example, that $[T_{0,\alpha}^{(k)},T_0,[T_\delta,T_\epsilon,T_0]] \neq 0.$ By lemmas \ref{lema0}, \ref{lema3} we have 
\[[T_{0,\alpha}^{(k)},T_0,[T_\delta,T_\epsilon,T_0]] \subseteq [T_{0,\alpha}^{(k)},T_0,T_{\delta+\epsilon}] \subseteq T_{\delta+\epsilon}\]
and by induction we have $\delta + \epsilon \subseteq [\alpha_0]\cup\{0\}.$ If any of $\delta, \epsilon$ is zero, then there is nothing more to prove. Suppose that both $\delta, \epsilon \neq 0.$ Since $[T_\delta,T_\epsilon,T_0] \neq 0,$ the lemma \ref{lema3} implies that $\delta \sim \epsilon.$ If $\delta + \epsilon \neq 0,$ then $\delta \sim \delta + \epsilon \sim \alpha_0$ and we are done. If $\delta = -\epsilon,$ then we have
\[0\neq [T_{0,\alpha}^{(k)},T_0,[T_\delta,T_{-\delta},T_0]] = [T_{0,\alpha}^{(k)},T_0,T_{0,\delta}^{(1)}],\]
which implies that $T(\alpha,\delta) \neq 0$ and $\alpha \sim \delta$ by lemma \ref{lema1}.

Moreover, if $\delta + \epsilon \neq 0,$ then the whole product lies in $T_{\delta + \epsilon} \subseteq V_{[\alpha_0]},$ and if $\delta + \epsilon = 0,$ then the whole product lies in $[T_{0,\alpha}^{(k)},T_0,T_0]  = T_{0,\alpha}^{(k+1)} \subseteq T_{0,[\alpha_0]}.$

Other nonzero products can be analyzed similarly.
\end{proof}



Let us denote
\[
T_{0,\Lambda^T}:= \sum\limits_{\tiny{\begin{array}{c}
  \alpha + \beta + \gamma=0 \\
\alpha, \beta \in
\Lambda^{T} \cup \{0\}\\
\gamma \in
\Lambda^{T} \end{array}}}[{T}_{\alpha},{T}_{\beta},{T}_{\gamma}] + \sum_{\alpha \in \Lambda^T}T_{0,\alpha}.
\]

Now we can prove the main results of this section:

\begin{theorem}\label{teo1}
Let $(T,\phi)$ be a Leibniz color $3$-algebra $T$ with an automorphism with root spaces decomposition $T= T_{0}\oplus \big(\bigoplus_{\alpha \in
\Lambda^T}T_{\alpha }\big)$. 
\begin{enumerate}
\item[{\rm (i)}]
For any $\alpha_{0} \in \Lambda^T$, the space
$T_{[{\alpha_0}]}$ is an ideal of $T$.
\smallskip
\item[{\rm (ii)}] If $T$ is simple, then any two roots in $\Lambda^T$ are connected. 
\end{enumerate}
\end{theorem}

\begin{proof}
(i) The spaces $T_{[\alpha]}$ are $\mathbb{G}$-homogeneous by construction. Recall that by definition of connectivity, any root $\beta \in \Lambda^T$ is connected to all $\beta\Phi^k, k \in \mathbb{Z}.$ Therefore, the lemma \ref{lema0} implies that for $\alpha \in [\alpha_0]$
\begin{equation}
\label{phi_root_eqv}
\phi^{\pm1}(T_\alpha) = T_{\alpha\Phi^{\mp1}} \subseteq V_{[\alpha_0]},
\end{equation}
thus $\phi^{\pm1}(V_{[\alpha_0]}) = V_{[\alpha_0]}.$
Moreover, the equation (\ref{phi_root_eqv}) and the definition of the space $T_{0,[\alpha_0]}$ easily imply that $\phi(T_{0,[\alpha_0]}) = T_{0,[\alpha_0]}.$ Hence, $\phi(T_{[\alpha_0]}) = T_{[\alpha_0]}.$ The rest of the  proof of the part (i) follows directly from lemmas \ref{lema3}, \ref{lema51}, \ref{lema52}. The part (ii) is obvious.
\end{proof}

\begin{theorem} \label{teo2}
If $\mathcal{U}$ is a vector space complement of $T_{0,\Lambda^T}$, then
$T = \mathcal{U} \oplus \sum_{[\alpha] \in \Lambda^T/\sim} T_{[\alpha]}$.
Moreover, $[T, T_{[\alpha]},T_{[\beta]}] + [T_{[\alpha]},T,T_{[\beta]}] + [T_{[\alpha]},T_{[\beta]}, T]= 0$,
whenever $[\alpha] \neq [\beta]$.
\end{theorem}

\begin{proof}
The first part of the theorem is obvious, so we only prove the second part. Let $\alpha_0, \beta_0 \in \Lambda^T$ such that $\alpha_0 \not\sim \beta_0$ and consider the product $[T, T_{[\alpha_0]},T_{[\beta_0]}].$ Lemmas \ref{lema3}, \ref{lema51}, \ref{lema52} imply that 
\[[T,V_{[\alpha_0]},V_{[\beta_0]}] + [T,T_{0,[\alpha_0]},V_{[\beta_0]}] + [T,V_{[\alpha_0]},T_{0,[\beta_0]}] = 0,\]
so we only need to prove that $[T,T_{0,[\alpha]},T_{0,[\beta]}] = 0.$ Moreover, lemmas \ref{lema51}, \ref{lema52} imply that $[T_\gamma,T_{0,[\alpha]},T_{0,[\beta]}] = 0$ for any $\gamma \in \Lambda^T,$ so it suffices to prove that $[T_0,T_{0,[\alpha]},T_{0,[\beta]}] = 0.$

Let $\alpha, \beta, \gamma \in [\alpha_0]\cup\{0\},$ such that $\alpha + \beta + \gamma = 0$ and $\delta, \epsilon, \chi \in [\beta_0]\cup\{0\}$ such that $\delta + \epsilon + \chi = 0.$ Suppose that $\gamma \neq 0$ (hence, at least one of $\alpha, \beta \neq 0$) By Leibniz identity, we have
\[[T_0,[T_\alpha,T_\beta,T_\gamma],T_{0,[\beta_0]}] \subseteq [T_\alpha,T_\beta,[T_0,T_\gamma,T_{0,[\beta_0]}]] + [[T_\alpha,T_\beta,T_0],T_\gamma,T_{0,[\beta_0]}] + [T_0,T_\gamma,[T_\alpha,T_\beta,T_{0,[\beta_0]}]] = 0\]
by lemmas \ref{lema51}, \ref{lema52}.
Analogously, we can show that if $\chi \neq 0,$ then
\[[T_0,T_{0,[\alpha]},[T_\delta,T_\epsilon,T_\chi]] = 0.\]
Therefore, it suffices to show that $[T_0,T_{0,\alpha}^{(k)},T_{0,\beta}^{(l)}] = 0$
for all $\alpha \in [\alpha_0], \beta \in [\beta_0], k, l \in \mathbb{N}.$ But if any such product is nonzero, then $T(\alpha,\beta) \neq 0$ and $\alpha \sim \beta$ by lemma \ref{lema1}, a contradiction.

All other products are considered completely similarly.

\end{proof}

As an obvious corollary, we get the following result:

\begin{corollary}
If $T_0 = T_{0,\Lambda^T}$ and $\ann(T) = 0,$ then $T = \bigoplus_{[\alpha]\in \Lambda^T/\sim}T_{[\alpha]}.$
\end{corollary}



\subsection{The $Hom$ case}

Now let us go back to the original setting. That is, let $(T,\phi)$ be a split regular $Hom$-Leibniz color 3-algebra. The results of the previous section imply that the algebra $T^{\phi^{-1}}$ is a split Leibniz color 3-algebra with the root system $\Lambda^{T^{\phi^{-1}}} = \Lambda^T\circ\varphi^{-1},$ where $\varphi: \mathcal{A}(T) \to \mathcal{A}(T^{\phi^{-1}})$ is an isomorphism of the canonical envelopes which restricts to an isomorphisim of $\mathfrak{L}(T)$ and $\mathfrak{L}(T^{\phi^{-1}})$. The main results of the last subsection are that there exists an equivalence relation $\sim_{T^{\phi^{-1}}}$ on the set $\Lambda^{T^{\phi^{-1}}}$ such that $T^{\phi^{-1}} $ is the direct sum of a subspace $\mathcal{U} \subseteq T_0$ and the sum of $\phi$-invariant ideals $T^{\phi^{-1}}_{[\alpha]}$ indexed by equivalence classes of $\sim_{T^{\phi^{-1}}}$ such that 
\[[T^{\phi^{-1}}, T^{\phi^{-1}}_{[\alpha]},T^{\phi^{-1}}_{[\beta]}]_{T^{\phi^{-1}}} + [T^{\phi^{-1}}_{[\alpha]},T^{\phi^{-1}},T^{\phi^{-1}}_{[\beta]}]_{T^{\phi^{-1}}} + [T^{\phi^{-1}}_{[\alpha]},T^{\phi^{-1}}_{[\beta]}, T^{\phi^{-1}}]_{T^{\phi^{-1}}}= 0,\]
whenever $[\alpha] \neq [\beta]$.

\medskip

Now we obtain an analogous decomposition for $T.$ First, the equation relating the multiplications in $T$ and $T^{\phi^{-1}}$ implies immediately  that the ideals $T^{\phi^{-1}}_{[\alpha]}, \alpha \in \Lambda^{T^{\phi^{-1}}}$ stay $\phi$-invariant ideals in $T,$ and
\[[T, T^{\phi^{-1}}_{[\alpha]},T^{\phi^{-1}}_{[\beta]}]_T + [T^{\phi^{-1}}_{[\alpha]},T,T^{\phi^{-1}}_{[\beta]}]_T + [T^{\phi^{-1}}_{[\alpha]},T^{\phi^{-1}}_{[\beta]}, T]_T= 0.\]
Having in mind the relation (\ref{dehom_root_spaces}), let us introduce an equivalence relation $\sim_T$ in $\Lambda^T$ saying that $\alpha \sim_T \beta$ for $\alpha, \beta \in \Lambda^T$ if and only if $\alpha\circ\varphi^{-1} \sim_{T^{\phi^{-1}}} \beta\varphi^{-1}.$ Therefore, for $\alpha \in \Lambda^{T^{\phi^{-1}}}$ we have $[\alpha\circ\varphi]_{\sim} = [\alpha]_{\sim_{T^{\phi^{-1}}}}\circ\varphi$ and
\[V_{[\alpha]} = \bigoplus_{\beta \in [\alpha]_{\sim_{T^{\phi^{-1}}}}} T_\beta^{\phi^{-1}} = \bigoplus_{\beta \in [\alpha]_{\sim_{T^{\phi^{-1}}}}} T_{\beta\circ\varphi} = \bigoplus_{\beta' \in [\alpha\circ\varphi]_{\sim}} T_{\beta'}.\]
Analogously one can show that
\[T_{0,[\alpha]} = {\rm span}_{\hu K}\{[T^{\phi^{-1}}_{\beta},T^{\phi^{-1}}_{\gamma},
T^{\phi^{-1}}_{\delta}]^{\phi^{-1}}: \beta, \gamma, \delta \in [\alpha]_{\sim_{T^{\phi^{-1}}}} \cup \{0\}, \beta + \gamma + \delta = 0, \delta \neq 0 \} + \sum_{\beta \in [\alpha]_{\sim_{T^{\phi^{-1}}}}} T^{\phi^{-1}}_{0,\beta} = \]
\[{\rm span}_{\hu K} \{[T_{\beta'},T_{\gamma'},
T_{\delta'}]: \beta', \gamma', \delta' \in [\alpha\circ\varphi]_{\sim_T} \cup \{0\}, \beta' + \gamma' + \delta' = 0, \delta'\neq 0 \} + \sum_{\beta' \in [\alpha\circ\varphi]_{\sim_T}} T_{0,\beta'},\]
where the spaces $T_{0,\beta'}$ are constructed in the same manner as in (\ref{t0alpha}). Therefore, for any $\alpha \in \Lambda^T$ we may define the ideal $T_{[\alpha]_{\sim_T}}$ by the same equations (\ref{t0[alpha]}), (\ref{v[alpha]}) and (\ref{t[alpha]}). Moreover, we have $T_{[\alpha]_{\sim_T}} = T_{[\alpha\circ\varphi^{-1}]_{\sim_{T^{\phi^{-1}}}}}.$

We state our result as a theorem:

\begin{theorem}
\label{decomp_hom}
Let $(T,\phi)$ be a split $Hom$-Leibniz color 3-algebra with multiplication algebra $\mathfrak{L}.$ Suppose that the root systems $\Lambda^T$ and $\Lambda^{\mathfrak{L}}$ are symmetric. Then there exists an equivalence relation $\sim$ on $\Lambda^T$ and a subspace $\mathcal{U} \subseteq T_0$ such that $T = \mathcal{U} \oplus \sum_{[\alpha] \in \Lambda^T/\sim} T_{[\alpha]}$.
Moreover, $[T, T_{[\alpha]},T_{[\beta]}] + [T_{[\alpha]},T,T_{[\beta]}] + [T_{[\alpha]},T_{[\beta]}, T]= 0$,
whenever $[\alpha] \neq [\beta]$.
\end{theorem}

Finally, note that the relation $\sim_T$ can be introduced in a manner completely analogous to the definition of the relation $\sim_{T^{\phi^{-1}}},$ that is, using the connections via chains of roots in $\Lambda^T$ (that is, it can be interpreted in the terms of the original algebra $T$).

\subsection{The passage of the split structure to $\mathfrak{L}$}

Let $T$ be a split $Hom$-Leibniz color 3-algebra or a split Leibniz color 3-algebra with an automorphism.  In this subsection we show that the decomposition provided by theorems \ref{teo2}, \ref{decomp_hom} induces a similar decomposition for the multiplication algebra $\mathfrak{L}(T),$ and that these two decompositions are related in a certain way. Moreover, we study how this decomposition is related to the usual decomposition of Lie color algebras obtained by usual root connectivity techniques.

\medskip

Let $(T,\phi)$ be a split Leibniz color 3-algebra with an automorphism, and $(\mathfrak{L},\Phi)$ is its multiplication algebra (the case of a split $Hom$-Leibniz color 3-algebra is completely analogous, so we will not consider it here). Again, suppose that both roots systems $\Lambda^T$ and $\Lambda^{\mathfrak{L}}$ are symmetric, so we that are in the conditions of the theorem \ref{teo2}.

For any equivalence class $[\alpha] \subseteq \Lambda^T$ consider the space $\ad(T_{[\alpha]}) = \ad(T_{[\alpha]},T) + \ad(T,T_{[\alpha]}) \subseteq \mathfrak{L}.$ Let also $\Lambda_{[\alpha]}^{\mathfrak{L}}$ be the set of roots $\gamma \in \Lambda^{\mathfrak{L}}$ such that $\mathfrak{L}_\gamma \cap \ad(T_{[\alpha]}) \neq 0.$

\begin{theorem} Let $T, \mathfrak{L}, \Lambda^T, \Lambda^{\mathfrak{L}}$ be as above. Then the following assertions hold:
\begin{enumerate}
    \item $\Lambda^{\mathfrak{L}}$ is the disjoint union of the sets $\Lambda_{[\alpha]}^{\mathfrak{L}},[\alpha] \in \Lambda^T/\sim.$ That is, there exists an equivalence relation $\approx$ in $\Lambda^{\mathfrak{L}}$ such that $\Lambda^T/\sim = \Lambda^{\mathfrak{L}}/\approx.$ 
    \item  Let $\gamma \in \Lambda_{[\alpha_0]}^{\mathfrak{L}}.$ Then  $\mathfrak{L}_\gamma \subseteq \ad(T_{[\alpha_0]}).$ In other words, for any $[\alpha_0] \in \Lambda^T$ we have 
    \[\ad(T_{[\alpha]}) = (\ad(T_{[\alpha]})\cap \mathfrak{L}_0) \oplus \bigoplus_{\gamma \in \Lambda_{[\alpha]}^{\mathfrak{L}}}\mathfrak{L}_{\gamma}.\]
    \item For any equivalence class $[\alpha] \in \Lambda^T/\sim$ the space $\ad(T_{[\alpha]})$ is a $\Phi$-invariant ideal of $\mathfrak{L}.$
    \item $[\ad(T_{[\alpha]}),\ad(T_{[\beta]})] = 0$ for $[\alpha] \neq [\beta].$
    \item There exists a subspace $\mathcal{U}' \subseteq \mathfrak{L}_0$ such that $\mathfrak{L} = \sum_{[\alpha]}\ad(T_{[\alpha]}) + \mathcal{U}'$
    \item For any $[\alpha] \neq [\beta]$ we have $\ad(T_{[\alpha]})\cdot T_{[\beta]} = 0.$
\end{enumerate}
\end{theorem}
\begin{proof}

\begin{enumerate}
\item Recall the decomposition of theorem \ref{teo2}: $T = \mathcal{U} \oplus \sum_{[\alpha] \in \Lambda^T/\sim} T_{[\alpha]}$ for $\mathcal{U} \subseteq T_0$ with $[T_{[\alpha]},T_{[\beta]},T] = 0$, whenever $[\alpha] \neq [\beta]$. Therefore, for any $\alpha \in \Lambda^T$ we have
\begin{equation}
\label{adtalpha}
\ad(T_{[\alpha]}) = \ad(T_{[\alpha]},T_{[\alpha]}) + \ad(T_{[\alpha]},\mathcal{U}) + \ad(\mathcal{U},T_{[\alpha]}) 
\end{equation}
and we obtain the following decomposition for $\mathfrak{L}:$
\begin{equation}
\label{L_decomp}
\mathfrak{L} = \ad(T,T) = \sum_{[\alpha]}\ad(T_{[\alpha]}) + \ad(\mathcal{U},\mathcal{U}).
\end{equation}

Since the ideals $T_{[\alpha]}$ are graded with respect to $\Lambda^T,$ lemma \ref{lema0} implies that the spaces $\ad(T_{[\alpha]}) \subseteq \mathfrak{L}$ are graded with respect to $\Lambda^{\mathfrak{L}},$ that is,
\begin{equation}
\label{adtalpha_graded}
\ad(T_{[\alpha]}) = \bigoplus_{\gamma \in \Lambda^{\mathfrak{L}} \cup \{0\}}(\ad(T_{[\alpha]}\cap \mathfrak{L}_\gamma).
\end{equation}
This, decomposition (\ref{L_decomp}) and the fact that $\ad(\mathcal{U},\mathcal{U}) \subseteq \mathfrak{L}_0$ imply that $\Lambda^{\mathfrak{L}} = \bigcup_{[\alpha] \subseteq \Lambda^T}\Lambda_{[\alpha]}^{\mathfrak{L}}.$ 
From relation (\ref{adtalpha}) it follows that
\[\Lambda_{[\alpha_0]}^{\mathfrak{L}} = \{\alpha+\beta : \alpha,\beta \in [\alpha_0]\cup\{0\}, \alpha + \beta \neq 0, \ad(T_\alpha,T_\beta) + \ad(T_\beta,T_\alpha) \neq 0\}\]
Suppose now that $\gamma \in \Lambda_{[\alpha_0]}^{\mathfrak{L}} \cap \Lambda_{[\beta_0]}^{\mathfrak{L}},$ that is,
$\gamma = \alpha_1 + \alpha_2 = \beta_1 + \beta_2$ \mbox{ where } $\alpha_1, \alpha_2, \beta_1, \beta_2$ satisfy the conditions above.
It is clear that at least one of the elements $\alpha_1, \alpha_2 \neq 0,$ so without loss of generality let $\alpha_1 \neq 0.$ Analogously let $\beta_1 \neq 0.$ Then one can easily see that $\{\alpha_1,\alpha_2,-\beta_2\}$ is a connection from $\alpha_1$ to $\beta_1,$ thus, $[\alpha_0] = [\beta_0].$ Therefore, any root $\gamma \in \Lambda_{\mathfrak{L}}$ belongs to a unique set $\Lambda_{[\alpha]}^{\mathfrak{L}}$ for a certain $[\alpha] \subseteq \Lambda^T$ and we $\Lambda^{\mathfrak{L}} = \bigsqcup_{[\alpha] \subseteq \Lambda^T}\Lambda_{[\alpha]}^{\mathfrak{L}}.$

This decomposition induces an equivalence relation $\approx$ in $\Lambda^{\mathfrak{L}}$ (two elements are equivalent if they lie in the same set $\Lambda_{[\alpha]}^{\mathfrak{L}}$).

\item The first claim is just a restatement of the assertion that the sets $\Lambda_{[\alpha]}^{\mathfrak{L}}, [\alpha] \in \Lambda^T/\sim$ have empty intersection, and the second is a consequence of the first and the $\Lambda^{\mathfrak{L}}$-homogeneity of $\ad(T_{[\alpha]}$ (\ref{adtalpha_graded}).

\item The $\Phi$-invariance is a consequence of the definition of $\Phi$ and the $\phi$-invariance of $T_{[\alpha]}.$ The rest is a simple application of the relation (\ref{AL1_color}). For example,
\[[\ad(T_{[\alpha]},T),\ad(T,T)] = \ad([T_{[\alpha]},T,T],T) + \ad(T,[T_{[\alpha]},T,T]) \subseteq \ad(T_{[\alpha]}).\]
Considering analogously the product $[\ad(T,T),\ad(T_{[\alpha]},T)],$ we are done.

\item Let $[\beta] \neq [\alpha].$ Again, by relation (\ref{AL1_color}),
\[[\ad(T,T_{[\alpha]}),\ad(T,T_{[\beta]})] = \ad([T,T_{[\alpha]},T],T_{[\beta]}) + \ad(T,[T,T_{[\alpha]},T_{[\beta]}]) = 0.\]
Other products are considered completely analogously.

\item In the decomposition (\ref{L_decomp}) we cannot guarantee that $\ad(\mathcal{U},\mathcal{U}) \cap \sum_{[\alpha]}\ad(T_{[\alpha]}) = 0,$ so just taking another subspace $\mathcal{U}' \subseteq \mathfrak{L}_0$ complementing $(\sum_{[\alpha]}\ad(T_{[\alpha]})) \cap \mathfrak{L}_0,$ we get
\[\mathfrak{L} = \sum_{[\alpha]}\ad(T_{[\alpha]}) + \mathcal{U}'.\]

\item This item follows easily from theorem \ref{teo2}.

\end{enumerate} 
\end{proof}

Therefore, for the algebra $\mathfrak{L}$ we get a decomposition that is very similar to the one in theorems \ref{teo2}, \ref{decomp_hom}, the only difference is that there does not seem to be an easy way to express the space $(\ad(T_{[\alpha]})\cap\mathfrak{L}_0$ in terms of the spaces $\mathfrak{L}_\gamma, \gamma \in \Lambda^{\mathfrak{L}}_{[\alpha]}.$ Another decomposition (by classical root connection methods) of split $Hom$-Lie color algebras was obtained in the paper \cite{Cao3}.

\begin{remark}
Recall that $\mathfrak{L}$ is a split $Hom$-Lie color algebra, and that $T$ is a weight $\mathfrak{L}$-module. Weight modules over split Lie algebras were considered in the paper \cite{AJCSmod}, in which the authors obtained related decompositions of a split Lie algebra $L$ and its weight module $M.$ 
\end{remark}

Note, however, that our decomposition only holds for split Lie color algebras with automorphism which are the multiplication algebras of Leibniz color 3-algebras with automorphism. But this class is relatively big, as shows the following assertion:

\begin{proposition}
Let $(\mathfrak{L},\Phi)$ be a split Lie color algebra with automorphism such that $\mathfrak{L} = \mathfrak{L}^2.$ Then there exists a split Leibniz color 3-algebra $T$ with automorphism $\phi$ such that its multiplication algebra is $\mathfrak{L}.$
\end{proposition}
\begin{proof}
Recall that any split Leibniz color 3-algebra $T$ with automorphism $\phi$ is a weight Lie color module over its multiplication algebra $\mathfrak{L},$ and the relation (\ref{DII}) is equivalent to the fact that the map $\ad: T^{\otimes 2} \to \mathfrak{L}$ is a homomorphism of $\mathfrak{L}$-modules. Clearly, by the construction of $\mathfrak{L},$ this homomorphism must be surjective. Moreover, the automorphisms $\phi$ and $\Phi$ are compatible in the sense that $\Phi(\ell)\cdot\phi(x) = \phi(\ell\cdot x)$ for all $\ell \in \mathfrak{L}, x \in T.$ In fact, one can easily check that the converse statement also holds, that is, to construct the desired algebra $T,$ it suffices to give a weight $\mathfrak{L}$-module $T$ with a surjective module homomorphism $\ad: T^{\otimes 2} \to \mathfrak{L}$ and an automorphism $\phi$ of $T$ compatible with $\Phi.$

So, we take $T = \mathfrak{L}$ with the regular module action, the automorphism $\phi = \Phi$ and the map $\ad: x \otimes y \mapsto [x,y],$ where $[\cdot,\cdot]$ is the multiplication in $\mathfrak{L}.$ This system satisfies the conditions above and is therefore a split Leibniz color 3-algebra $T$ (in fact, a Lie triple system) with automorphism  with the multiplication algebra $\mathfrak{L}.$
\end{proof}

\medskip

Given a split Lie color algebra $L$ with automorphism $\varphi$ and the (symmetric) root system $\Lambda^L$ one can introduce the notion of roots connectivity in $L$ (see, for example, paper \cite{AJC} and others on binary split algebras):

\begin{definition}
\label{L_connect}
Say that two roots $\alpha, \beta \in \Lambda^L$ are connected (and denote it as $\alpha \sim \beta$) if there exists a chain $\{\alpha_1,\alpha_2,\ldots,\alpha_n\} \subseteq \Lambda^L$ such that 
\begin{enumerate}
\item $\alpha_1 = \alpha.$
\item $\alpha_1 + \ldots + \alpha_k \in \Lambda^L$ for any $k = 1,\ldots,n.$
\item $\alpha_1 + \ldots + \alpha_n = \pm\beta\varphi^k,$ where $k \in \mathbb{Z}.$
\end{enumerate}
\end{definition}

One can see that $\sim$ is an equivalence relation on $\Lambda^L,$ and that there exists a subspace $\mathcal{V} \subseteq L_0$ such that $L = V \oplus \bigoplus_{[\gamma] \in \Lambda^L/\sim}L_{[\gamma]},$ where $L_{[\gamma]}$ are $\varphi$-invariant ideals such that $[L_{[\gamma]},L_{[\delta]}] = 0$ for $[\gamma] \neq [\delta].$ 

Now let us return to the algebra $\mathfrak{L}.$ There are now two equivalence relations on the root system $\Lambda^{\mathfrak{L}}$: usual (intrinsic) root connectivity given by the definition above and the relation $\approx$ which appears by ''descending'' the decomposition of $T$ to $\mathfrak{L}.$  It would be interesting to see how these two equivalence relations are related. For now, we have the following partial result.

\begin{definition}
Let $L$ be a split Lie color algebra $L$ with automorphism $\varphi$ and the root system $\Lambda^L.$ We say that $L$ is root-multiplicative if for any two roots $\alpha, \beta \in \Lambda^L$ such that $\alpha + \beta \in \Lambda^L \cup \{0\}$ we have $[L_\alpha,L_\beta] \neq 0.$
\end{definition}

\begin{proposition}
Let $\mathfrak{L}$ be root-multiplicative and let $\gamma_1, \gamma_2 \in \Lambda^{\mathfrak{L}}$ be two connected roots. Then $\gamma_1 \approx \gamma_2.$
\end{proposition}
\begin{proof}
We proceed by induction. Let first $\gamma_2 = \pm\gamma_1\Phi^k$ for $k \in \mathbb{Z}.$ By definition of $\mathfrak{L},$ we have 
\[\gamma_i = \alpha_i + \beta_i, \mbox{ where } \ad(T_{\alpha_i},T_{\beta_i}) + \ad(T_{\beta_i},T_{\alpha_i}) \neq 0, \alpha_i, \beta_i \in \Lambda^T \cup\{0\}.\]
Without loss of generality let $\alpha_1, \alpha_2 \neq 0.$
Then $\{\alpha_1,\beta_1,\mp\beta_2\Phi^k\}$ is a connection from $\alpha_1$ to $\alpha_2,$ and $\gamma_1 \approx \gamma_2.$

\medskip

Now, let $\gamma_1 + \gamma_2 = \pm\gamma_3\Phi^k,$ where $\gamma_3 = \alpha_3 + \beta_3$ for $\alpha_3, \beta_3 \in \Lambda^{T}\cup\{0\}.$ By root-multiplicativity of $\mathfrak{L},$ we have
\[[\mathfrak{L}_{\gamma_1},\mathfrak{L}_{\gamma_2}] \neq 0,\]
so
\[[\ad(T_{\alpha_1},T_{\beta_1}),\ad(T_{\alpha_2},T_{\beta_2})] \neq 0\]
for some $\alpha_i,\beta_i \in \Lambda^T\cup\{0\}$ such that $\alpha_i + \beta_i = \gamma_i.$
By Leibniz identity (\ref{AL1_color}), we get
\[0 \neq [\ad(T_{\alpha_1},T_{\beta_1}),\ad(T_{\alpha_2},T_{\beta_2})] \subseteq \ad([T_{\alpha_1},T_{\beta_1},T_{\alpha_2}],T_{\beta_2}) + \ad([T_{\alpha_1},T_{\beta_1},T_{\beta_2}],T_{\alpha_2}),\]
thus, $[T_{\alpha_1},T_{\beta_1},T_{\alpha_2}] + [T_{\alpha_1},T_{\beta_1},T_{\beta_2}] \neq 0.$
Without loss of generality suppose that  $[T_{\alpha_1},T_{\beta_1},T_{\alpha_2}] \neq 0,$ which implies that $\alpha_1 + \beta_1 + \alpha_2 \in \Lambda^T \cup \{0\}.$ One of the $\alpha_1, \beta_1$ is nonzero. Suppose that $\alpha_1 \neq 0,$ the case $\beta_1 \neq 0$ is completely analogous.

Suppose that $\alpha_1 + \beta_1 + \alpha_2 \in \Lambda^T.$ Then it is easy to see that the set $\{\alpha_1,\beta_1,\alpha_2,\beta_2,\mp\beta_3\Phi^k\}$ is a connection from $\alpha_1$ to $\alpha_3,$ and that the set $\{\alpha_1,\beta_1,-(\alpha_1 + \beta_1 + \alpha_2)\}$ is a connection from $\alpha_1$ to $\alpha_2.$

Now, if $\alpha_1 + \beta_1 + \alpha_2 = 0,$ then $\alpha_1 + \beta_2 = -\alpha_2 \in \Lambda^T \cap \Lambda^{\mathfrak{L}}$ and $\{\alpha_1,\beta_1,0\}$ is a connection from $\alpha_1$ to $\alpha_2.$ Moreover, $\beta_2 = \pm(\alpha_3 + \beta_3)\Phi^k = \pm\gamma_3\Phi^k \in \Lambda^T\cap\Lambda^{\mathfrak{L}}$ and $\{\beta_2,0,\mp\beta_3\Phi^k\}$ is a connection from $\beta_2$ to $\alpha_3.$ Therefore, all nonzero roots $\alpha_i, \beta_i$ are connected and $\gamma_1 \approx \gamma_2 \approx \gamma_3$.

\end{proof}

It would also be interesting to express the relation $\approx$ in terms of strings of roots in $\lambda^T, \Lambda^{\mathfrak{L}}$, such as in the definition \ref{L_connect}, but as of now the authors do not now a way to do so.

{\bf Acknowledgment.}
The authors would like to thank Prof. Dr. Antonio Jes\'us Calder\'on Mart\'in (University of C\'adiz, Spain) for many useful coments and discussions about the paper.


\end{document}